\documentclass[final,onefignum,onetabnum]{siamart190516}
\usepackage[margin=1.33in]{geometry}

\usepackage[T1]{fontenc}
\usepackage[utf8]{inputenc}
\usepackage[shortlabels]{enumitem}
\usepackage{mathtools,amsfonts}
\usepackage{hyperref}
\usepackage[capitalize]{cleveref}
\usepackage{autonum}
\usepackage{xcolor}
\usepackage{subcaption}
\usepackage{graphicx}
\usepackage{url}
\usepackage[output-decimal-marker={.}]{siunitx}
\usepackage{tabularx,booktabs}
\usepackage{dsfont}
\usepackage[misc]{ifsym}
\usepackage{cite}

\newcommand{\scal}[2]{\left\langle{#1},{#2}  \right\rangle}
\newcommand{\norm}[1]{\left\lVert#1\right\rVert}
\newcommand{\abs}[1]{\left\lvert#1\right\rvert}

\newcommand{\BB}{\ensuremath{\mathcal B}}
\newcommand{\Ncone}{\ensuremath{\mathcal N}}
\newcommand{\Tcone}{\ensuremath{\mathcal T}}
\newcommand{\RR}{\ensuremath{\mathds R}}
\newcommand{\NN}{\ensuremath{\mathds N}}

\newcommand{\DDiag}{\ensuremath{\mathbf D}}
\newcommand{\XX}{\ensuremath{\mathbf X}}
\newcommand{\RP}{\ensuremath{\mathds R}_+}

\newcommand{\closu}{\ensuremath{\operatorname{cl}}}

\DeclareMathOperator{\epi}{epi}

\DeclareMathOperator{\conv}{conv}
\DeclareMathOperator{\cone}{cone}

\DeclareMathOperator{\Fix}{Fix}
\DeclareMathOperator{\Id}{Id}

\DeclareMathOperator{\dist}{dist}
\DeclareMathOperator{\bap}{bap}

\newcommand{\supH}{\mathds{H}}

\usepackage{calrsfs}
\DeclareMathAlphabet{\mathcal}{OMS}{zplm}{m}{n}

{\begin{list}{}{%
\settowidth{\labelwidth}{\textrm{#1~}}%
\setlength{\leftmargin}{\labelwidth+\labelsep}}}
{\end{list}}

\renewcommand\theenumi{(\roman{enumi})}
\renewcommand\theenumii{(\alph{enumii})}

\usepackage[shortlabels]{enumitem}
\newlist{lista}{enumerate}{1}
\setlist[lista]{label=\alph*., nosep,leftmargin=*,align=right}

\newlist{listi}{enumerate}{1}
\setlist[listi]{label={\upshape(\roman*\upshape)},leftmargin=*,align=right, widest=iii,nosep, format=\bf}

\allowdisplaybreaks

\newsiamremark{example}{Example}
\newsiamremark{remark}{Remark}
\crefname{hypothesis}{Hypothesis}{Hypotheses}
\newsiamthm{fact}{Fact}
\crefname{fact}{Fact}{Facts}

\headers{Finite convergence of alternating projections}{R. Behling, Y. Bello-Cruz  and L.-R. Santos}

\title{Infeasibility and error bound imply finite convergence of alternating projections\thanks{
\funding{ RB was partially supported by the \emph{Brazilian Agency  Conselho Nacional de Desenvolvimento Cient\'ifico e Tecnol\'ogico} (CNPq), Grants 304392/2018-9 and 429915/2018-7;  YBC was partially supported by the \emph{National Science Foundation} (NSF), Grant DMS — 1816449.}}}

\author{Roger Behling\thanks{School of Applied Mathematics, Funda\c{c}\~ao Getulio Vargas 
Rio de Janeiro, RJ — 22250-900, Brazil. (\email{rogerbehling@gmail.com})~\Letter}
\and 
Yunier Bello-Cruz\thanks{Department of Mathematical Sciences, Northern Illinois University.   DeKalb, IL — 60115-2828, USA. (\email{yunierbello@niu.edu})} 
\and 
Luiz-Rafael Santos\thanks{Department of Mathematics, Federal University of Santa Catarina. 
Blumenau, SC — 88040-900, Brazil. (\email{l.r.santos@ufsc.br})}
}

\begin{document}

\maketitle

\begin{abstract} 
This paper combines two ingredients in order to get a rather surprising result on one of the most studied, elegant and powerful tools for solving convex feasibility problems, the method of alternating projections (MAP). Going back to names such as Kaczmarz and von Neumann, MAP has the ability to track a pair of points realizing minimum distance between two given closed convex sets. Unfortunately, MAP may suffer from arbitrarily slow convergence, and sublinear rates are essentially only surpassed in the presence of some Lipschitzian error bound, which is our first ingredient.  The second one is a seemingly unfavorable and unexpected condition, namely, infeasibility. For two non-intersecting closed convex sets satisfying an error bound, we establish finite convergence of MAP. In particular, MAP converges in finitely many steps when applied to a polyhedron and a hyperplane in the case in which they have empty intersection. Moreover, the farther the target sets lie from each other, the fewer are the iterations needed by MAP for finding a best approximation pair. Insightful examples and further theoretical and algorithmic discussions accompany our results, including the investigation of finite termination of other projection methods. 



\end{abstract}

\section{Introduction}\label{sec:intro}  

The method of alternating projections (MAP) has a remarkable impact in so many areas of Mathematics and is one of the main classical tools for solving convex feasibility problems. A broad class of problems in Applied Mathematics is effectively solved by MAP~\cite{Deutsch:1992}. Definitely one of a kind, MAP is not only capable of tracking a point in the intersection of given closed convex sets $X,Y\subset \RR^{n}$, it delivers a replacement of an actual solution when $X$ and $Y$ do not intersect. Such a replacement comes in the form  of a pair $(\bar x, \bar y)\in X\times Y$, often called \emph{best approximation pair} to $X$ and $Y$, as it minimizes the Euclidean distance between these two sets. 

It is well-known that MAP converges globally whenever the distance between $X$ and $Y$ is attainable. It is also worth mentioning that, in view of Pierra's famous product space reformulation \cite{Pierra:1984}, feasibility problems involving a finite number of sets can be narrowed down to seeking a common point to two sets $X$ and $Y$.

The present work focuses on the inconsistent case $X\cap Y=\emptyset$ and reveals a surprising behavior of MAP in this setting. Roughly speaking, we come to the conclusion that infeasibility works in favor of MAP. Quite intuitive when looking at the scenarios displayed in \Cref{fig:MAPfigure1}, the fact that infeasibility has a strong positive impact on MAP has apparently not been seen anywhere in its extensive literature. Actually, we prove that infeasibility added by an error bound condition provides finite convergence of the very pure MAP. More precisely, by assuming $X\cap Y=\emptyset$ and a suitable error bound condition, we prove finite convergence of the MAP sequence defined by $x^{k+1}\coloneqq P_XP_Y(x^k)$ starting at any point $x^0\in \RR^{n}$. Finite convergence means that for some non-negative integer $\bar k$, MAP reaches a best point $\bar x\in X$, that is, $x^{\bar k}=\bar x$ and $\bar y = P_Y(\bar x)$ form a best approximation pair to $X$ and $Y$. Here and throughout the text, $P_X$ and $P_Y$ stand for the orthogonal projections onto $X$ and $Y$, respectively.

Let us now look at a collection of illustrations that serves as a summary of our results. 
\begin{figure*}[!htp]\centering
  \begin{subfigure}[b]{.48\linewidth}
    \centering
    \includegraphics[width=.9\textwidth]{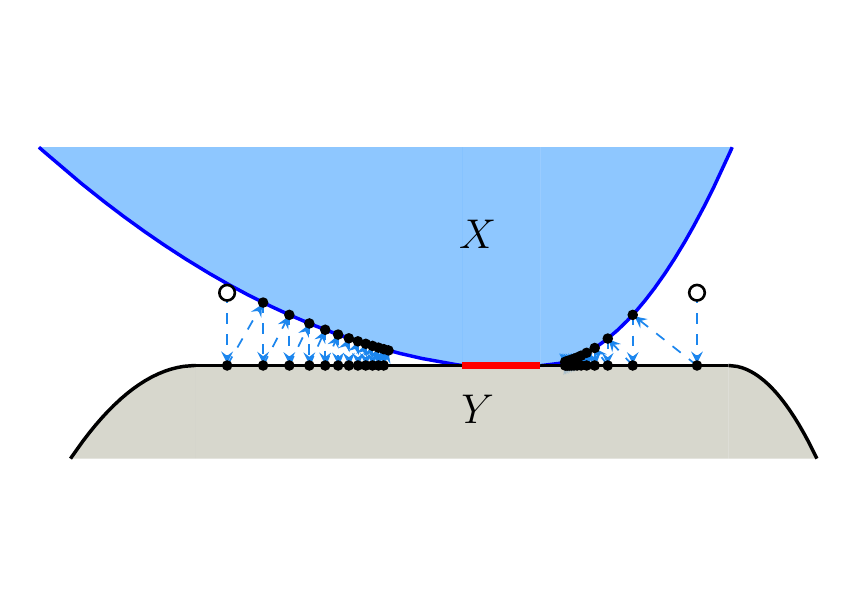}
    \caption{\label[figure]{fig:MAPfigure1a} Consistency; linear convergence on the left.}
  \end{subfigure}
  \begin{subfigure}[b]{.48\linewidth}
    \centering
    \includegraphics[width=.9\textwidth]{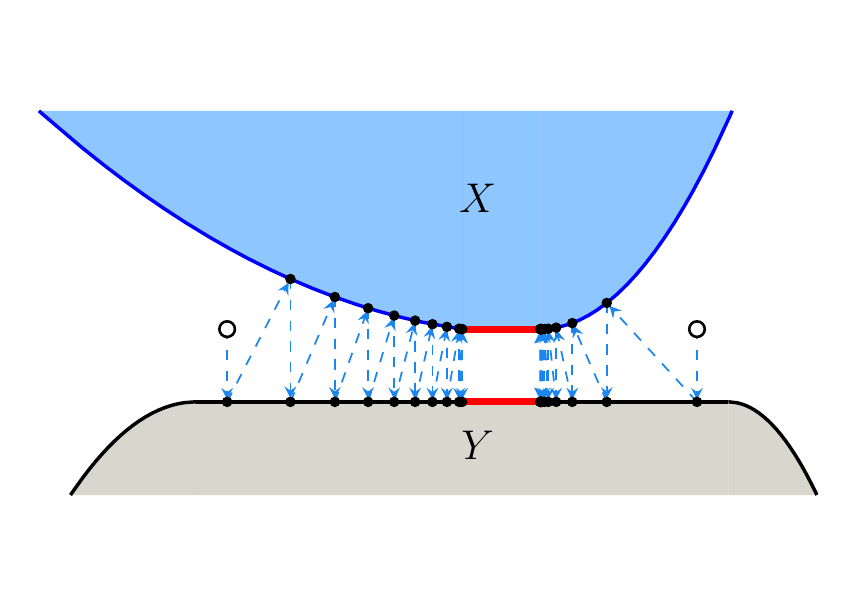}
    \caption{\label[figure]{fig:MAPfigure1b} Inconsistency; $k=10$ on the left.}
  \end{subfigure} \\
  \begin{subfigure}[t]{.48\linewidth}
    \centering
    \includegraphics[width=.9\textwidth]{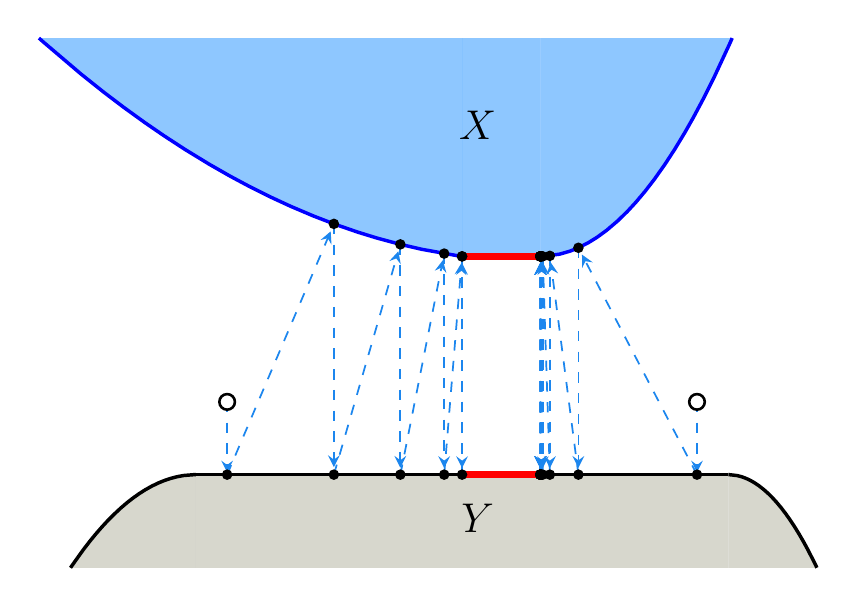}
    \caption{Inconsistency; increasing the distance between $X$ and $Y$; $k=5$ on the left.\label[figure]{fig:MAPfigure1c}}
  \end{subfigure}
  \begin{subfigure}[t]{.48\linewidth}
    \centering
    \includegraphics[width=.9\textwidth]{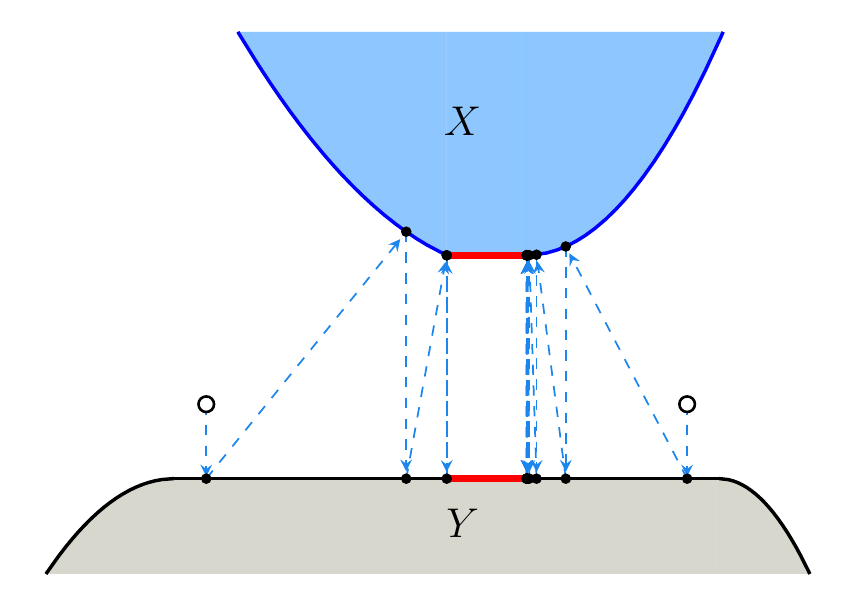}
    \caption{Inconsistency; improving the error bound  on the left; $k=3$  on the left.\label[figure]{fig:MAPfigure1d}}
  \end{subfigure}
  \caption{Error bound on the left-hand sides.}
   \label{fig:MAPfigure1}
\end{figure*}

\Cref{fig:MAPfigure1} displays four scenarios of MAP acting on sets $X$ and $Y$. We start with a consistent problem in  \cref{fig:MAPfigure1a} and analyze two MAP sequences. The one starting from the left converges linearly to a point in $X\cap Y$. This is due to the fact that in this region $X$ and $Y$ form a non-zero  angle. In other words, a Lipschitzian error bound holds. This is not the case on the right-hand side of the picture and therefore, MAP only achieves sublinear convergence over there. By lifting $X$, we generate inconsistent intersection problems in \labelcref{fig:MAPfigure1b,fig:MAPfigure1c,fig:MAPfigure1d}. The result is that MAP responds favorably speed-wise to this translation of $X$. More than that, MAP improves when increasing infeasibility and also when the error bound gets better, that is, the left border of $X$ gets steeper. This can be noticed looking at \Cref{fig:MAPfigure1} as a film from \cref{fig:MAPfigure1a,fig:MAPfigure1b,fig:MAPfigure1c,fig:MAPfigure1d}. Lifting $X$ from \cref{fig:MAPfigure1a} to \cref{fig:MAPfigure1b} makes MAP's convergence jump from linear  to finite on the left. On the right-hand side, MAP leaps its convergence rate from sublinear to linear. After a further lift of $X$ from \cref{fig:MAPfigure1b} to \cref{fig:MAPfigure1c}, MAP reaches a best point in $5$ iterations instead of $10$ on the left-hand side. MAP only needs $3$ iterates to get a best approximation pair when improving the error bound in the left part of \cref{fig:MAPfigure1d}.

The message of \Cref{fig:MAPfigure1} is fairly clear. As for the aforementioned error bound condition, it will be formally introduced later. We anticipate, though, that it regards the target sets and their relation with what we are going to call optimal supporting hyperplanes. We point out that such an error bound is automatically globally fulfilled if $X$ is a polyhedron and $Y$ a hyperplane, providing finite convergence of MAP if, in addition, these particular polyhedral sets have empty intersection. Note that finite convergence of MAP may happen even under a non-polyhedral structure, as depicted in \Cref{fig:MAPfigure1}.

Before outlining the structure of our paper, we briefly go through some turning points in the history of MAP. This method has fascinated scientists for nearly a century now, and although many results on this simple tool have been derived, open questions remain. In the early 1930s, von Neumann firstly studied MAP for subspaces. Yet, his results were only published in 1950~\cite{Neumann:1950}, proving the convergence of MAP to a so-called best approximation solution for the intersection of two subspaces. In 1937, Kaczmarz~\cite{Kaczmarz:1937} proposed a MAP related algorithm (also known as Cyclic projections) for finding best approximation solutions of linear systems. In 1959, MAP was studied for the convex case by Cheney and Goldstein~\cite{Cheney:1959}, covering also inconsistent feasibility problems. Precisely the theme of our paper, projection methods for inconsistent inclusions have a history on their own; see the 2018 review by Censor and Zaknoon~\cite{Censor:2018}. The ability of MAP to find best approximation pairs is a notable characteristic, making MAP (and its variants) one of the most used algorithms in Optimization~\cite{ Deutsch:1995, Deutsch:1997,Coakley:2011, Bauschke:1995, Censor:1991, Diaconis:2010,Bauschke:2016a,Smith:1977,Behling:2020b,DePierro:1988,Aharoni:2018}. 
Although MAP always
converges under the existence of best approximation pairs, the convergence rate may be arbitrarily slow~\cite{Bauschke:2009, Badea:2012, Franchetti:1986}. In 1950, Aronszajn~\cite{Aronszajn:1950}  found the lower bound for the linear rate of MAP, given by the square of the cosine of the minimal angle (Friedrichs angle) between two subspaces, which turns out to be the sharpest one, as proved by Kayalar and Weinert~\cite{Kayalar:1988} in 1988. 
For an in depth related literature on MAP; see, for instance, Bauschke and Borwein~\cite{Bauschke:1993,Bauschke:1996}, and Deutsch~\cite{Deutsch:1992}. 

The paper is organized as follows. In \Cref{sec:facts} we collect known facts on MAP and some auxiliary material.  \Cref{sec:error-bound} contains our mayor results concerning the geometry of two disjoint closed convex sets in the presence of an error bound condition. This analysis allows us to derive in \Cref{sec:finitemethods} our main result, namely,  finite convergence of MAP under inconsistency and what we call BAP error bound. The BAP error bound is automatically fulfilled for a polyhedron and a hyperplane, and we study its connection with linear regularity and intrinsic transversality. Also in \Cref{sec:finitemethods}, we will see that for a MAP sequence to converge in a finite number of steps, it is necessary and sufficient that it satisfies the BAP error bound at its iterates. \Cref{sec:finitemethods} ends by investigating finite termination under inconsistency of Cyclic projections, Cimmino and Douglas-Rachford methods.  In \Cref{sec:discussion}, we present insightful examples and applications. In particular, we connect our results to Linear Programming and convex min-max problems. We consider as well a simple problem giving rise to the question on whether a H\"older type error bound could make MAP's rate of convergence go from sublinear to linear when shifting the target sets apart. Some concluding remarks are presented in~\Cref{sec:concluding}.


\section{Background material}\label{sec:facts}

Let  $X,Y\subset \RR^{n}$ be closed, convex and nonempty. We recall that the orthogonal projection of $x\in \RR^n$ onto $X$ is given  by  $P_{X}(x)\in X$ if, and only if, $\scal{y-P_{X}(x)}{x - P_{X}(x)}\le 0,$ for all $y\in X$. Throughout the text, $\scal{\cdot}{\cdot}$ stands for the Euclidean inner product inducing the norm $\norm{\cdot}\coloneqq\sqrt{\scal{\cdot}{\cdot}}$. The non-negative integer numbers will be denoted by $\NN$. The open ball centered in $x$ with radius $\delta >0$ is the set $\BB_{\delta}(x)\coloneqq \{z\in\RR^{n}\mid \norm{x - z} <\delta\}$.
 
We define the \emph{distance} between $X$ and $Y$ by
$\label{eq:dist}
\dist(X,Y) \coloneqq \inf \{\norm{x-y}\mid x\in X, y\in Y\}.$
When one of the sets is a singleton, for instance $X = \{x\}$, we use the notation $\dist(x,Y)$.   A \emph{best approximation pair} (BAP)  relative to $X$ and $Y$ is a  pair $(\bar x,\bar y)\in X\times Y$ attaining  the distance between $X$ and $Y$, that is,  $\dist(\bar x,\bar y) = \dist(X,Y)$.  The set of all BAP relative to $X$ and $Y$ is denoted by $\bap(X,Y)\subset X\times Y $ and, accordingly, we define the  sets \[\bap_{Y}(X) \coloneqq\{x\in X\mid (x,y)\in \bap(X,Y)\}=\{x\in X\mid \dist(x,Y)=\dist(X,Y)\}\] and \[\bap_{X}(Y) \coloneqq\{y\in Y\mid (x,y)\in \bap(X,Y)\}=\{y\in Y\mid \dist(y, X)=\dist(X,Y)\}.\] Note that $\bap_{Y}(X)$ (respectively $\bap_{X}(Y)$) is the subset of points in $X$  (respectively $Y)$ \emph{nearest} to $Y$ (respectively $X$). 
In the consistent case, that is,  when $X\cap Y$ is nonempty, we have $\bap_{X}(Y) = \bap_{Y}(X) = X\cap Y$. As we are interested in the inconsistent case, henceforth, we suppose that $X\cap Y = \emptyset$.  In this context we define the \emph{ displacement vector} as $d\coloneqq P_{\closu(X-Y)}(0).$ So, $\norm{d}=\dist(X,Y)$ and $\dist(X,Y)$ is attained if, and only, if $d\in X-Y$. In particular, $\dist(X,Y)$ is attained whenever $X-Y$ is closed.

Given any point $x\in \RR^{n}$, define the terms of the sequence $(x^{k})_{k\in\NN}$ by
\[\label{eq:MAPsequence}
x^0 = x,\quad x^{k+1} = P_XP_Y(x^{k}),
\]
for every $k\in\NN$. The sequence $(x^{k})_{k\in\NN}$ is the \emph{alternating projection sequence} starting at $x^{0}=x$. Cheney and Goldstein established in~\cite{Cheney:1959}  that if one of the sets is compact or if one of the sets is finitely generated, the  fixed point set of the operator $P_XP_Y$ is nonempty and the sequence \cref{eq:MAPsequence} converges to a fixed point of this operator. The general result was summarized and enlarged in  \cite{Bauschke:1994} as follows. 
\begin{fact}[BAP sets {\cite[Lemma 2.2]{Bauschke:1994}}]\label{fact:bapsets}
Denote $\Fix(P_XP_Y)=\{x\in \RR^n \mid P_XP_Y(x)=x\}$. Then,
\begin{listi}
  \item $\bap_{Y}(X) = \Fix(P_XP_Y)$.
  \item  $\bap_{Y}(X)$ and $\bap_{X}(Y)$ are closed convex sets.
  \item  If $\bap_{Y}(X)$ or $\bap_{X}(Y)$ is nonempty then $\dist(X,Y)$ is attained. Moreover, let $d$ be the displacement vector. Then
\[P_Y(\bar x)=\bar x-d, \quad \forall \bar x\in \bap_{Y}(X),\]
  and 
 $\bap_{Y}(X)-d=\bap_{X}(Y), \bap_{Y}(X)=X\cap (Y+d), \bap_{X}(Y)=(X-d)\cap Y.$ 
\end{listi}
\end{fact}

In the next fact, we abuse notation and use $\scal{X}{y}\le 0$ to denote that $\scal{x}{y}\le 0,\forall x\in X$. 

\begin{fact}[BAP pairs]\label{fact:dist-attained}
Let $x\in\RR^{n}$ be given. Then, if $\dist(X,Y)$ is attained, with  $d$ being the displacement vector, then$$P_{\bap_{X}(Y)}(x)=P_{\bap_{Y}(X)}(x)-d,$$ and $\langle X-\bap_{Y}(X),d\rangle\ge 0$ and $\langle Y-\bap_{X}(Y), d\rangle \le 0$.
\end{fact}
\begin{fact}[Convergence of MAP {\cite[Theorem 4.8]{Bauschke:1994}}]\label{conv-MAP} Let $(x^k)_{k\in\NN}$ be an alternating projection sequence as given in \cref{eq:MAPsequence}. Then,  \[x^k-P_Y(x^k)\to d,\] where $d$ is the displacement vector.  Moreover, 
\begin{listi}
\item if $\dist(X, Y)$ is attained then   $x^k\to \bar x\in \bap_{Y}(X)$ and $P_Y(x^k)\to \bar y\coloneqq \bar x-d\in \bap_{X}(Y)$;
\item  if $\dist(X, Y)$ is not attained then $\norm{x^k} \to +\infty$.
\end{listi}
\end{fact}

Next we present some definitions and well-known results concerning polyhedra, useful in \Cref{thm:SingleStepMAP-PolyHyp,thm:FiniteMAP-poly}.
\begin{definition}[Polyhedron]
A set $\Omega\subset \RR^{n}$ is said to be a  \emph{polyhedron}, if it can be expressed as the intersection of a finite family of closed half-spaces, that is, 
\[\label{eq:polyhedron}
  \Omega \coloneqq \{x \in \RR^{n} \mid \scal{a_i}{x} \leq \alpha_i,  \text{ for } i=1,\ldots, m\},
\]
where $ a_{i}\in\RR^{n}, \alpha_{i}\in\RR$.
\end{definition}

Note that a polyhedron, also referred to as a \emph{polyhedral set}, is always convex.

\begin{definition}[Finitely generated cone and conic base]\label{def:cones} A set $K\subset\RR^{n}$ is  a \emph{cone} if it is closed under positive scalar multiplication.  A cone $K$ is  \emph{finitely generated} if there exists a finite set $S\subset \RR^{n}$ such that $\cone(S) = K$, where $\cone(S)$, is the set of all conic combinations of elements of $S$. A \emph{conic base} of a finitely generated cone $K$ is a finite set $B_{K}\subset\RR^{n}$ with minimal cardinality such that $\cone(B_{K}) = K$.

\end{definition}

 \begin{fact}[Polyhedron is finitely generated {\cite[Proposition B.17]{Bertsekas:1999}}]\label{fact:finite-generators}
A set $\Omega\subset \RR^{n}$ is polyhedral if, and only if, it is finitely generated, \emph{i.e.},  there exist a nonempty and finite set of vectors $\left\{v_{1}, \ldots, v_{m}\right\},$ and a finitely generated cone $K$ such that
\[
\Omega=\{x \in \RR^{n} \mid x=y+\sum_{j=1}^{m} \mu_{j} v_{j}, y \in K, \sum_{j=1}^{m} \mu_{j}=1, \mu_{j} \geq 0, j=1, \ldots, m\}.
\]
\end{fact}

Let us note that finitely generated cones  are the same as polyhedral cones, because of the well-known Minkowski-Weyl Theorem~\cite[Theorem 3.52]{Rockafellar:2004}.

\begin{definition}[Tangent and Normal cones]
  Let $X$  be a nonempty closed convex set in $\RR^{n}$ and $ x\in X$. The \emph{tangent cone} of  $X$ at $ x$ is given by 
    \[
      \Tcone_{X}( x)\coloneqq \closu\left(\left\{ \lambda(y- x) \in \RR^{n}\mid y \in X, \lambda \in \RP\right \}\right).
    \]
    The \emph{normal cone} of $X$ at $x$ is the set defined by 
    \[
      \Ncone_{X}( x)\coloneqq \left\{ y \in \RR^{n}\mid \scal{y}{z-x}\leq 0 , z \in X\right \}.
    \]
    \end{definition}

\begin{fact}[Tangent cone of polyhedron {\cite[Theorem 6.46]{Rockafellar:2004}}]\label{fact:tangent-cone-poly}
If $\Omega\subset \RR^{n}$ is a polyhedron defined as in \cref{eq:polyhedron}, then   the tangent cone $\Tcone_{\Omega}(x)$, at any point $x\in \Omega$, is a polyhedral cone and can be represented as 
\[
\Tcone_{\Omega}(x) = \{w \in \RR^{n} \mid \scal{a_i}{w}\le 0, \text{ for } i \in \mathcal{I}(x)\},
\]
where $a_i\in\RR^{n}$ defines the polyhedron $\Omega$ and $\mathcal{I}(x) \coloneqq\{i  \mid \scal{a_i}{x} = \alpha_i \}$ is the active index set of $\Omega$ at $x$.
\end{fact}

\begin{fact}[Finitely many tangent cones of a polyhedron~\cite{Bertsekas:1999}]\label{cor:finitetangentcones} If $\Omega$ is a polyhedron, then the set of all tangent cones  $\{\Tcone_{\Omega}(x)\mid x\in \Omega\}$ has finite cardinality. Moreover, for any $x\in\Omega$, there exists a radius $\delta>0$ such that $(\Tcone_{\Omega}(x) +x) \cap \mathcal{B}_{\delta}(x)=\Omega \cap \mathcal{B}_{\delta}(x)$, that is, $\Omega$ coincides locally with any shifted  tangent cone to it.
\end{fact}

Finally, it is noteworthy that the distance between two  non-intersecting polyhedral sets is attained. 

\begin{fact}[Distance between polyhedra is attained~{\cite[Theorem 5]{Cheney:1959}}]\label{fact:polyMAP} 
Let $\Omega_1,\Omega_2\subset\RR^{n}$ be nonempty polyhedra with $\Omega_1\cap\Omega_2=\emptyset$. Then, $
\dist(\Omega_1,\Omega_2)$ is attained.
\end{fact} 

\section{On the geometry of two disjoint convex sets under error bound condition}\label{sec:error-bound}

This section is divided in two subsections, gathering key contributions of our paper. In the first subsection, we investigate the geometry under which a single alternating projection step reaches a best approximation pair when the sets are apart from each other. In the second one, we compare well-known regularity conditions from the literature with our proposed error bound. 

\subsection{BAP-EB and alternating projection step}
Here we start stating that a single alternating projection step yields a best approximation pair to a polyhedron $\Omega$ and a hyperplane $H$, if $\Omega\cap H=\emptyset$ and  the alternating projection step is taken from a point sufficiently close to $\bap_H(\Omega)$.

\begin{lemma}[Alternating projection step for polyhedron versus hyperplane under inconsistency]\label{thm:SingleStepMAP-PolyHyp}
  Consider two nonempty sets $\Omega,H\subset \RR^{n}$ such that $\Omega$ is a polyhedron, $H$ is a hyperplane and $\Omega\cap H = \emptyset$. Then, there exists a radius $r>0$ such that  
  \[\label{eq:SingleStepMAP-PolyHyp}
    P_{\Omega}P_{H}(z) \in \bap_H(\Omega), 
  \]
  for all $z\in\RR^n$ satisfying $\dist(z,\bap_H(\Omega))\leq r$.
\end{lemma}

\begin{proof}  For a point $x\in \Omega$, let $\Tcone_{\Omega}(x)$ denote the tangent cone of $\Omega$ at $x$. Since $\Omega$ is a polyhedron, the set of all tangent cones  $\{\Tcone_{\Omega}(x)\mid x\in \Omega\}$ is finite and each $\Tcone_{\Omega}(x)$ is a finitely generated cone (see \Cref{fact:finite-generators,fact:tangent-cone-poly,cor:finitetangentcones}).  In particular, the cardinality of $\Gamma\coloneqq\{\Tcone_{\Omega}(x)\mid x\in \bap_{H}(\Omega)\}$ is finite since $\bap_{H}(\Omega)\subset \Omega$ and each tangent cone in $\Gamma$ has a finite number of generators. Consider now the collection of all normalized generators with respect to cones in $\Gamma$ denoted by $W\coloneqq \{w \in \RR^{n} \mid \norm{w}= 1, \text{ such that} \; w \text{ belongs to a conic  base of some } \Tcone_{\Omega}(x), x\in \bap_{H}(\Omega)   \}$. 

Let $d$ be the displacement vector. \Cref{fact:dist-attained} allows us to conveniently categorize the generators in $W$.  For the disjoint finite sets $U\coloneqq\{u\in W\mid \norm{u}=1, \scal{u}{d}=0\}$ and $V\coloneqq\{v\in W\mid \norm{v}=1,  \scal{v}{d} > 0\}$, we have  $W = U\cup V$. The  finiteness and definition of $V$ provide the existence and positivity of  
\[\label{eq:defr} r\coloneqq\min_{v\in V}\{\scal{v}{d}\}.\] 
The facts that $r>0$ and $\norm{d}>0$ are key for the statement  \eqref{eq:SingleStepMAP-PolyHyp}.

  Take $z\in \RR^{n}$ arbitrary, but fixed, such that $\dist(z,\bap_{H}(\Omega))\leq r$. In order to shorten the notation, set $\bar z\coloneqq P_{\bap_{H}(\Omega)}(z)$ and $z_{H}\coloneqq P_{H}(z)$.
Let $y$ be an element of $\bap_H(\Omega)$. \cref{fact:bapsets}(iii) implies that $ y-d$ and $\bar{z}-d$ are contained in $\bap_\Omega(H)\subseteq H$. Since $z_H\in H$ and $H$ is affine,
$ y-\bar{z}+z_H = (y-d)-(\bar{z}-d)+z_H \in H.$
The  fact that $H$ is a hyperplane  and the characterization of the best approximation for $\Omega$ and $H$ imply
\begin{align}
\langle z_H-\bar{z},y-\bar{z}\rangle
&= \langle z-\bar{z},y-\bar{z}\rangle + \langle z_H-z,(y-\bar{z}+z_H)-z_H\rangle = \langle z-\bar{z},y-\bar{z}\rangle  \leq  0.
\end{align}
Thus, $\bar{z}=P_{\bap_H(\Omega)}(z_H)$.



  Let us now look at the angle between $z_{H}-\bar z$ and vectors in $\Tcone_{\Omega}(\bar z)$. Since $\bar z$ is in $\bap_{H}(\Omega)$, all the generators of the tangent cone $\Tcone_{\Omega}(\bar z)$ must be contained in $W$. Recall that $W$ is split as $U\cup V$ and, therefore, in order to investigate the sign of   $\scal{w}{z_{H}-\bar z}$,
  consider the two cases below:
  \begin{enumerate}[(a), nosep, format=\bf,leftmargin=*,align=right]
  \item $w$ is a generator of $\Tcone_{\Omega}(\bar z)$ belonging to $U$;

  \item $w$ is a generator of $\Tcone_{\Omega}(\bar z)$ belonging to $V$.
  \end{enumerate}

Case (a). For $w$ to be a generator of $\Tcone_{\Omega}(\bar z)$ belonging to $U$,  it must be a generator of $\Tcone_{\bap_{H}(\Omega)}(\bar z)$, since $\bap_{H}(\Omega) \subset H+d$. Bear in mind that a polyhedron coincides locally with its tangent cone (See \Cref{cor:finitetangentcones}). So, for some radius $\varepsilon>0$, $\bap_{H}(\Omega)\cap \BB_{\varepsilon}(\bar z) = (\Tcone_{\bap_{H}(\Omega)}(\bar z) + \bar{z})\cap  \BB_{\varepsilon}(\bar z) $. Thus, for all $t>0$ sufficiently small, $\bar z + tw\in \bap_{H}(\Omega)$ and since $P_{\bap_{H}(\Omega)}(z_{H}) = \bar z$, by projections onto convex sets, we have  $\scal{\bar z + tw - \bar z}{z_{H}-\bar z} \leq 0$, therefore  $\scal{w}{z_{H}-\bar z} \leq 0$.

Case (b). This is the case in which the constant $r>0$ defined in \cref{eq:defr} is going to be employed. We now have $w\in V$,  
\begin{align}
\scal{w}{z_{H}-\bar z} &= \scal{w}{z_{H}-\bar z  + d} - \scal{w}{d}  \\
              &\leq \norm{w}\norm{z_{H} + d  - \bar z }- \scal{w}{d}\\
              &\leq  \norm{z -\bar z }- \scal{w}{d} \\
              &\leq r-\scal{w}{d}\leq 0,
\end{align}
where we used Cauchy-Schwarz  in the first inequality, the fact that $w$ is a unit vector and the Pythagoras argument  $ \norm{z_{H} + d  - \bar z }^2 =\norm{z -\bar z }^2  - \norm{z_{H} + d - z}^2  $ in the second one, and the definition of $r$ given in \cref{eq:defr} in the last one. 

Hence, cases (a) and (b) have shown that $\scal{z_{H}-\bar z}{w}\leq 0$ for any normalized generator $w$ of the cone $\Tcone_{\Omega}(\bar z)$. This means that the projection of $z_{H}$ onto the shifted cone $\Tcone_{\Omega}(\bar z)+\bar z$ is given by $\bar z$. Since $\bar z\in \Omega$ and $\Omega\subset \Tcone_{\Omega}(\bar z)+\bar z$, we get that $P_{\Omega}(z_{H}) = \bar z$. Recalling that we set $z_{H} = P_H(z)$, the proof is finished.
\end{proof}

We remark that replacing the hyperplane $H$ by an arbitrary polyhedron in the previous lemma  does not, in general,  guarantee the statement \cref{eq:SingleStepMAP-PolyHyp}; see \cref{ex:lackoferrorbound}. The key for reaching a best point in a single alternating step essentially relies on the existence of a suitable error bound   between the two non-intersecting sets, which in the case of \Cref{thm:SingleStepMAP-PolyHyp} is automatically fulfilled. 

The geometrical appeal of \Cref{thm:SingleStepMAP-PolyHyp} will inspire us to formulate the so-called \emph{BAP error bound}, which depends on the  definition of optimal supporting hyperplane, next.  


\begin{definition}[Optimal supporting hyperplane]\label{def:optimalhyperplane} Let $X,Y\subset\RR^{n}$ be closed convex sets such that $X\cap Y=\emptyset$ and that $\dist(X,Y)$ is attained. We say that
\[\supH_{Y}(X) \coloneqq \left\{z \in \RR^n \mid \scal{z - \bar x}{d} = 0, \bar x\in \bap_Y(X)\right\}\] 
is the \emph{optimal supporting hyperplane to $X$ regarding $Y$}, where $d$ is the displacement vector.  
\end{definition}
\begin{figure}[!htpb]
  \centering
  \includegraphics[width=.50\textwidth]{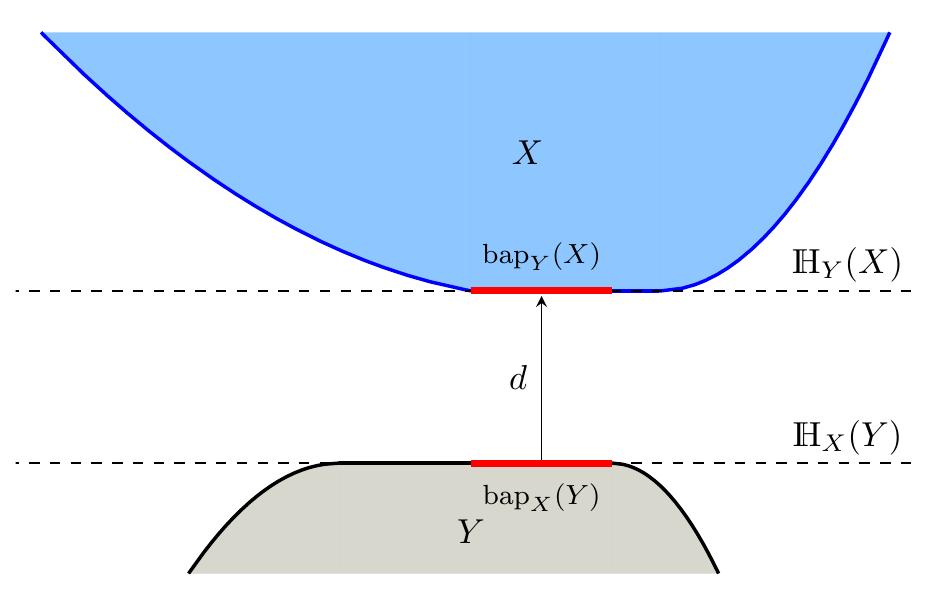}
  \caption{Disjoint convex sets $X$ and $Y$, BAPs and optimal hyperplanes.} 
  \label{fig:BAP}
\end{figure}

Note that $\supH_{Y}(X)$ is well-defined since any points $\hat x, \bar x\in \bap_{Y}(X)$ satisfy  $\scal{\hat x- \bar x}{d}=0$. \Cref{def:optimalhyperplane} together with \Cref{fact:bapsets}(iii) imply that, for two disjoint closed convex sets $X,Y\subset \RR^{n}$ with attainable distance, $\bap_{Y}(X)  = X \cap (Y+d) \subset \supH_Y(X)$ and  $\bap_{X}(Y)  = (X-d)\cap Y  \subset \supH_X(Y)$. \Cref{fig:BAP} illustrates all these sets related to $X$ and $Y$ as well as the displacement vector $d$.

We now formally present the BAP error bound in a non-symmetrical version.

\begin{definition}[Unilateral BAP error bound]\label{def:BAP-error-bound}
  Let $X, Y \subset \RR^{n}$ be non-intersecting  closed convex sets and assume that the distance between them is attained. We say that $X$ and $Y$ satisfy the \emph{unilateral BAP error bound (unilateral BAP-EB)}   at $x^{*}\in  \bap_Y(X)$ if there exist a bound  $\omega > 0$ and a radius $\delta > 0$ such that  the following inequality holds
  \begin{equation}\label{eq:error-bound}
  \omega \dist(x,\bap_{Y}(X)) \leq \dist(x, \supH_{Y}(X)), \text{ for all } x\in \BB_{\delta}(x^{*})\cap X.
  \end{equation}
  \end{definition}

\cref{def:BAP-error-bound} is a condition where the error bound is concentrated in only one of the two sets, therefore, the term \emph{unilateral}. Nevertheless, the responsibility of carrying an error bound can be distributed between the sets $X$ and $Y$. 
In this regard, we will introduce  in \cref{def:BAP-error-bound-II} a more general error bound condition, in which a bilateral concern about the error bound is embedded.


Error bounds are, in general, regularity conditions that allow one to deal with non-isolated solutions; see, for example,~\cite{Behling:2019,Behling:2013b}. The BAP-EB  resembles the well-known concept of local linear regularity\cite{Bauschke:1993,Bauschke:1996}, formally presented in \Cref{def:linear-regularity}. Linear regularity  is also known as subtransversality~\cite{Kruger:2018}. BAP-EB requires, depending on the context, a further geometrical feature between the sets, in comparison to linear regularity.
 In \Cref{sec:BAP-EB},  we present a detailed discussion on BAP-EB in which we compare it with the standard linear regularity (subtransversality) and with another regularity condition known as intrinsic transversality~\cite{Drusvyatskiy:2015,Drusvyatskiy:2016}.

Along this subsection, we are going to derive two important lemmas in which we consider the unilateral BAP-EB. They will be valuable tools towards establishing finite convergence of MAP in \Cref{sec:finitemethods} under the more general error bound condition BAP-EB that we introduce later in \cref{def:BAP-error-bound-II}.  

We present next a result encompassing a broader class of instances than the one in \cref{thm:SingleStepMAP-PolyHyp}, since we are going to consider a closed convex set $X$ and a hyperplane $H$ with empty intersection and attainable distance. Here, the optimal supporting hyperplane $\supH_{H}(X)$ to $X$ regarding $H$ coincides precisely with  $H+d$, that is, the hyperplane obtained by shifting $H$ by the displacement vector $d$. In this context, unilateral BAP-EB from \Cref{def:BAP-error-bound} is equivalent to local linear regularity. We prove this equivalence in \cref{prop:EB-linear-regularity}. Taking this equivalence into account, the next result states that the standard linear regularity condition leads to finite convergence of alternating projections when the target convex sets are disjoint and one of them is a hyperplane.



\begin{lemma}[Alternating projection step for a convex set versus hyperplane under inconsistency]
\label{thm:SingleStep-ConvexHyperplane}
Let $X, H \subset \RR^{n}$ be  a closed convex set and a hyperplane, respectively, and suppose that $X$ and  $H$ are disjoint with attainable distance. Assume that $X$ and $H$ satisfy the unilateral BAP error bound \cref{eq:error-bound} from \Cref{def:BAP-error-bound}  at  $x^{*}\in \bap_H(X)$, that is,  
 \[\label{eq:error-bound-condXH}
 \omega \dist(x,\bap_{H}(X)) \leq \dist(x, \supH_{H}(X)), \text{ for all } x\in \BB_{\delta}(x^{*})\cap X,
 \]
 with bound  $\omega > 0$ and  radius $\delta > 0$.
By setting $r \coloneqq \min\left\{\omega \dist(X,H),\frac{\delta}{2}\right\}$, we have that 
\[P_{X}P_{H}(z) \in \bap_{H}(X),\text{ for all } z\in  \BB_{r}(x^{*}).\]

\end{lemma}

\begin{proof}

 Take $z\in  \BB_{r}(x^{*})$, arbitrary, but fixed, with $x^*$ and $r$ as enunciated in the hypothesis. Set $\bar z \coloneqq P_{\bap_{H}(X)}(z)$ and $z_{H}\coloneqq P_{H}(z)$ and define
\[
  S\coloneqq \{s\in \RR^{n} \mid \scal{s-\bar z}{z_{H} - \bar z} \leq 0\}. 
\]
Since, $X\cap H=\emptyset$, we have that $z_{H} - \bar z\neq 0$ and thus, $S$ is an affine half-space. The keystone of the proof is to show that 
\[\label{eq:OmegasubsetV}
  X\subset S.
\] 

If this claim is proved, we get, directly from the characterization of a projection onto a closed convex set, that  $\bar z = P_{X}(z_{H}) = P_{X}P_{H}(z)$. Therefore, let us draw our attention to proving that \cref{eq:OmegasubsetV} holds. Assume the contrary, that is, there exists a point $w\in X$ which does not lie in $S$. Then, for a sufficiently small $t\in(0,1]$, we have $x\coloneqq tw + (1-t)\bar z\in   \BB_{\delta}(x^{*}) \cap X$. In fact, $x\in X$, by convexity. Now, we have that
\begin{align}
\norm{x-x^{*}} & = \norm{tw+(1-t)\bar z - x^{*}} =\norm{t(w-x^{*})+(1-t)(\bar z - x^{*})} \\
               &\leq  t \norm{w-x^{*}} + (1-t)\norm{\bar z - x^{*}}\\
               &=  t \norm{w-x^{*}} + (1-t)\norm{P_{\bap_{H}(X)}(z) - P_{\bap_{H}(X)}(x^{*})}\\
               & \leq t \norm{w-x^{*}} + (1-t)\norm{z - x^{*}} \\
               & < t \norm{w-x^{*}} + (1-t)r \\
               &  \leq t \norm{w-x^{*}} + (1-t)\frac{\delta}{2}\\
               & < t \norm{w-x^{*}} + \frac{\delta}{2}, 
 \end{align}   
where we used the definition of $x$, the convexity of the norm, the definition of $\bar z$ and the fact that $x^{*} \in \bap_{H}(X)$, the nonexpansiveness of the projection, the assumption $z\in \BB_{r}(x^{*})$, the definition of $r$, and the fact that $1-t  < 1$, respectively. Let us take a fixed $t \in \left(0,\min\left\{1,\frac{\delta}{2\norm{w-x^{*}}}\right\}\right]$. Thus,  the correspondent $x = tw + (1-t)\bar z \in\BB_{\delta}(x^{*})\cap X$. Note also that
\begin{align}
   \scal{x - \bar z}{z_{H} - \bar z} & = \scal{tw + (1-t)\bar z  - \bar z}{z_{H} - \bar z}  \\
        & = t\scal{w - \bar z}{z_{H} - \bar z} \\
        &>0,
\end{align}
since $t>0$ and $w\notin S$. Hence, $x\notin S$.

We proceed by showing that this point $x$ does not comply with the error bound condition \cref{eq:error-bound}, leading to a contradiction.

Simple Pythagoras arguments imply that
$P_{\bap_{H}(X)}(z_{H}) = \bar z \coloneqq P_{\bap_{H}(X)}(z) $, as we will see next. Let $z_{\supH} \coloneqq P_{\supH_{H}(X)}(z) $, where $\supH_{H}(X)$ is the optimal supporting hyperplane to $X$ regarding $H$ introduced in \Cref{def:optimalhyperplane}. Reminding that, in this case, $\supH_{H}(X) = H+d$, we get $z_{\supH} = z_{H} + d$.
Then,
\begin{align}\norm{z -  \bar z}^{2}  - \norm{\bar z - z_{\supH}}^{2}  & =
  \norm{ z - z_{\supH}}^{2}  
  = \norm{z -  P_{\bap_{H}(X)}(z_{\supH})}^{2}  - \norm{P_{\bap_{H}(X)}(z_{\supH}) - z_{\supH}}^{2} \\
  & \geq \norm{z -  \bar z}^{2}  - \norm{P_{\bap_{H}(X)}(z_{\supH}) - z_{\supH}}^{2}.
\end{align}
Crossing out $\norm{P_{\bap_{H}(X)}(z_{\supH}) - z_{\supH}}^{2}$ yields  $\norm{\bar z - z_{\supH}} \leq \norm{P_{\bap_{H}(X)}(z_{\supH}) - z_{\supH}}$, which gives us $P_{\bap_{H}(X)}(z_{\supH}) = \bar z$. Similarly,   considering the Pythagoras relations, we have
\begin{align}
\norm{z_{H} -  \bar z}^{2}  - \norm{\bar z - z_{\supH}}^{2}  & =
  \norm{ z_{H} - z_{\supH}}^{2}  
  = \norm{z_{H} -  P_{\bap_{H}(X)}(z_{H})}^{2}  - \norm{P_{\bap_{H}(X)}(z_{H}) - z_{\supH}}^{2} \\
  & \leq \norm{z_{H} -  P_{\bap_{H}(X)}(z_{H})}^{2}  - \norm{\bar z - z_{\supH}}^{2},
\end{align}
because $\norm{\bar z - z_{\supH}} \leq \norm{P_{\bap_{H}(X)}(z_{H}) - z_{\supH}}$. After a cancellation, we get  \[\norm{ z_{H} - \bar z } \leq  \norm{z_{H} -  P_{\bap_{H}(X)}(z_{H})},\] providing  $P_{\bap_{H}(X)}(z_{H}) =    \bar z$.

The fact that $P_{\bap_{H}(X)}(z_{H})  = \bar z $ implies that $\scal{s - \bar z}{z_{H} - \bar z} \leq 0,$ for all $s\in \bap_{H}(X)$, and we conclude by definition of $S$ that $\bap_{H}(X) \subset S$. On the other hand, recall that $\bap_{H}(X) \subset \supH_{H}(X)$ (see  \Cref{def:optimalhyperplane}). In particular, \[\label{eq:bapOmegainScapHd}
\bap_{H}(X)\subset S\cap \supH_{H}(X).
\] 

 Now, let us  define $\hat x \coloneqq P_{S\cap \supH_{H}(X)}(x)$. Since $x\notin S$, we have  $\scal{x - \bar z}{z_{H} - \bar z } > 0$ and because $\hat x$ lies on the   affine half-space $S$, it holds that   $ \scal{\hat x - \bar z}{z_{H} - \bar z} \leq  0$.
 Hence, 
 \[\scal{x - \hat  x}{z_{H} - \bar z} = \scal{x - \bar z}{z_{H} - \bar z} + \scal{\bar z - \hat x}{z_{H} - \bar z}> 0.\]
  Moreover, 
\begin{align}
   0 & < \scal{x-\hat x}{z_{H} - \bar z} =  \scal{x-\hat x}{ z_{H}  -  (\bar z - d)   - d } \\ 
        & =  \scal{x-\hat x}{z_{H}  -  (\bar z - d) }  - \scal{x-\hat x}{d}  \\
        & \leq  \norm{x-\hat x}\norm{z_{H}  + d - \bar z  }- \scal{x-\hat x}{d} \\
        & = \norm{x-\hat x}\norm{P_{\supH_{H}(X)}(z) - P_{\supH_{H}(X)} (\bar z) }- \scal{x-\hat x}{d}\\ 
        & \leq \norm{x-\hat x}\norm{z - \bar z }- \scal{x-\hat x}{d}\\ 
         & <  \norm{x-\hat x}r - \scal{x-\hat x}{d}, \label{eq:FiniteMAP-Convex-1}
\end{align}
where the second inequality is by Cauchy-Schwarz, the third is by the nonexpansiveness of projection   and the last one follows from $\norm{z-\bar z} \leq \norm{z - x^{*}}< r$. Therefore, 
\[\label{eq:FiniteMAP-Convex-2}
\scal{x-\hat x}{d} < r\norm{x-\hat x} \leq  \omega\norm{d}\norm{x-\hat x},\] 
by \cref{eq:FiniteMAP-Convex-1} and the definition of $r$. 

Now,  let $x_{H} \coloneqq P_{H}(x) $ and $x_{\supH} \coloneqq P_{\supH_{H}(X)}(x)  = x_{H} + d$. Then, 
\[\label{eq:FiniteMAP-Convex-3}
\scal{x-\hat x}{d} = \scal{x-x_{\supH} + x_{\supH} - \hat x}{d} = \scal{x-x_{\supH}}{d} + \scal{x_{\supH} - \hat x}{d}  = \scal{x-x_{\supH}}{d},\]as $\hat x, x_{\supH}\in \supH_{H}(X)$ and $d\perp \supH_{H}(X)$. Due to the fact that $\supH_{H}(X)$ is a hyperplane, $x-x_{\supH}$ is collinear to $d$, so  Cauchy-Schwarz holds sharply, that is,   $\abs{\scal{x-x_{\supH}}{d}} = \norm{x-x_{\supH}}\norm{d}$. Moreover, since $x\in X$, the inner product $\scal{x-x_{\supH}}{d}$ is nonnegative and thus $\scal{x-x_{\supH}}{d} = \norm{x-x_{\supH}}\norm{d}$, which combined with \cref{eq:FiniteMAP-Convex-2,eq:FiniteMAP-Convex-3}
provides
\[
\norm{x-x_{\supH}}\norm{d} < \omega\norm{d}\norm{x-\hat x}.
\]
This inequality, together with  $\bap_{H}(X)\subset  S\cap \supH_{H}(X)$, as proved in \cref{eq:bapOmegainScapHd},   yields
\begin{align}
\dist(x,\supH_{H}(X))  & = \norm{x-x_{\supH}}  \\
& < \omega \norm{x-\hat x} =   \omega \dist(x,S\cap \supH_{H}(X)) \\
& \leq   \omega \dist(x,\bap_{H}(X)),
\end{align}
which  contradicts the error bound assumption \cref{eq:error-bound-condXH}, because $x \in\BB_{\delta}(x^{*})\cap X$ and $\bap_{H}(X) = X\cap \supH_{H}(X)$.
\end{proof}

The previous result leads to another contribution of this paper. We show that the two ingredients, infeasibility and unilateral BAP-EB, imply that a single alternating projection step can locally reach a best approximation pair.

\begin{lemma}[Alternating projection step under infeasibility and unilateral BAP-EB]\label{thm:SingleStep-Convex}
Let $X, Y \subset \RR^{n}$ be  closed convex sets such that $X\cap Y = \emptyset$ and assume that the distance between them is attained.  Assume that $X$ and $Y$ satisfy the unilateral BAP error bound \cref{eq:error-bound} from \Cref{def:BAP-error-bound}  at  $x^{*}\in \bap_Y(X)$, that is,  
 \[\label{eq:error-bound-cond}
 \omega \dist(x,\bap_{Y}(X)) \leq \dist(x, \supH_{Y}(X)),\text{ for all } x\in \BB_{\delta}(x^{*})\cap X,
 \]
 with bound $\omega > 0$ and radius $\delta > 0$. By setting $r \coloneqq \min\left\{\omega \dist(X,Y),\frac{\delta}{2}\right\}$ we have that for all $z\in  \BB_{r}(x^{*})$,
\[P_{X}P_{Y}(z) \in \bap_{Y}(X).\]

\end{lemma}

\begin{proof}

Consider $x^*$ and $r$ as stated in the assumptions and let $z\in  \BB_{r}(x^{*})$ be arbitrary, but fixed. Now, set $z_{Y}\coloneqq P_{Y}(z)$ and $z^{\diamond} \coloneqq P_{\supH_{X}(Y)}(z_{Y})$, where $\supH_{X}(Y)$ is the optimal supporting hyperplane to $Y$ regarding $X$, \emph{i.e.}, $\supH_{Y}(X) = \supH_{X}(Y) + d$, where $d$ is the displacement vector. Moreover, using  the nonexpansiveness of projection operators onto convex sets and that $x^{*} - d$ lies in both $Y$ and $\supH_{X}(Y)$, we obtain
\begin{align}
\norm{(z^{\diamond} + d) - x^{*}} & = \norm{z^{\diamond} - (x^{*} - d)}\\
               & = \norm{P_{\supH_{X}(Y)}(z_{Y}) - P_{\supH_{X}(Y)}(x^{*} - d)}\\   
               & \leq \norm{z_{Y} - (x^{*} - d)} \\ 
               & =  \norm{P_{Y}(z) - P_{Y}(x^{*})} \\
               & \leq \norm{z - x^{*}} \leq r,  \label{eq:zdiamondinball}
\end{align}
that is, $z^{\diamond} + d\in \BB_{r}(x^{*})$.
Taking into account that $\dist(X,\supH_{X}(Y)) = \dist(X,Y) = \norm{d}$,  \Cref{thm:SingleStep-ConvexHyperplane} can be applied to $z^{\diamond} + d$, with $\supH_{X}(Y)$  playing the role of $H$, yielding $ P_{X}P_{\supH_{X}(Y)}(z^{\diamond}+ d) \in \bap_{\supH_{X}(Y)}(X)$. 
Note that $ \bap_{\supH_{X}(Y)}(X) \subset X\cap \supH_{Y}(X)$ and clearly,  the fact that the error bound constant $\omega$ in \cref{eq:error-bound-cond} is strictly positive implies that $\bap_{Y}(X) \cap \BB_{r}(x^{*}) =(X\cap\supH_{Y}(X))\cap \BB_{r}(x^{*})$. 
Thus, 
\[\label{eq:zdiamondinbap}
P_{X}(z^{\diamond}) \in \bap_{Y}(X),
\]
since  $P_{X}(z^{\diamond}) \in \BB_{r}(x^{*})$. This occurs because $P_{X}(z^{\diamond})= P_{X}P_{\supH_{X}(Y)}(z^{\diamond} + d)$, $x^{*}\in X\cap \supH_{Y}(X)$,  the nonexpansiveness of projections and \cref{eq:zdiamondinball} as 
\begin{align}
\norm{P_{X}(z^{\diamond}) - x^{*}} & = \norm{P_{X}P_{\supH_{X}(Y)}(z^{\diamond} + d) - P_{X}P_{\supH_{X}(Y)}(x^{*})} \\
& \leq  \norm{P_{\supH_{X}(Y)}(z^{\diamond} + d) - P_{\supH_{X}(Y)}(x^{*})} \\
& \leq  \norm{(z^{\diamond} + d) - x^{*}} \\
&\leq r.
\end{align}

Bearing in mind the definition of $\supH_{X}(Y)$ and that $z_{Y}\in Y$, observe that there exists $t\leq 0$ such that  $z_{Y} - z^{\diamond} = td$. For all $x\in X$, we have
\begin{align}
\scal{z_{Y} - P_{X}(z^{\diamond})}{x - P_{X}(z^{\diamond})} & = \scal{z_{Y} - z^{\diamond}}{x - P_{X}(z^{\diamond})} + \scal{z^{\diamond} - P_{X}(z^{\diamond})}{x - P_{X}(z^{\diamond})} \\ 
& = \underbrace{t}_{\leq 0} \underbrace{\scal{d}{x - P_{X}(z^{\diamond})}}_{\geq 0} + \underbrace{\scal{z^{\diamond} - P_{X}(z^{\diamond})}{x - P_{X}(z^{\diamond})}}_{\leq 0} \\ 
& \leq 0,\label{eq:zYiszdiamond}
\end{align}
where the second under-brace remark is by \Cref{fact:dist-attained} and the third is by characterization of the orthogonal projection of $z^{\diamond}$ onto $X$. Hence, \cref{eq:zYiszdiamond} implies that $P_{X}(z_{Y}) = P_{X}(z^{\diamond})$. Therefore, due to \cref{eq:zdiamondinbap} and that  $P_{X}P_{Y}(z) = P_{X}(z_{Y})$, it holds that
\[
P_{X}P_{Y}(z)  \in \bap_{Y}(X),
\]
proving the theorem.
\end{proof}


\subsection{BAP-EB and other regularity conditions}\label{sec:BAP-EB}
We now proceed to discuss the connection of BAP-EB with two other regularity conditions: local linear regularity, which is the keystone for providing linear convergence of MAP~\cite[Corollary 3.14]{Bauschke:1993}, and intrinsic transversality. 

First, we show that unilateral BAP-EB implies the standard local linear regularity and moreover, both coincide in the context of \Cref{thm:SingleStep-ConvexHyperplane}, that is, when one of the sets is a hyperplane. This is proved in \cref{prop:EB-linear-regularity} below.  This subsection ends with a proof that BAP-EB is more general than intrinsic transversality.
For clarity, we present the definitions of linear regularity and intrinsic transversality.

\begin{definition}[{Local linear regularity~\cite[Definition~3.11]{Bauschke:1993}}]\label{def:linear-regularity}
Let $X, Y \subset \RR^{n}$ be  closed convex sets such that $X\cap Y = \emptyset$, assume that the distance between them is attained and suppose that $d$ is the displacement vector.
We say that $X$ and $Y$ are \emph{locally linearly regular} if for some point $x^*\in\bap_{Y}(X) $ there exist  a bound $\kappa > 0$ and a radius $\rho >0$ and such that, for all $z\in \BB_{\rho}(x^{*})$,  
\[\label{eq:linear-regularity}
\dist(z,\bap_{Y}(X)) \leq \kappa  \max \{\dist(z, X), \dist(z, Y+d)\}.
\]
\end{definition}

\begin{proposition}\label{prop:EB-linear-regularity}
Let $X, Y \subset \RR^{n}$ be  closed convex sets, suppose that $X$ and  $Y$ are disjoint with attainable distance and let $d$ be the displacement vector. Then, the unilateral BAP-EB condition from \cref{def:BAP-error-bound} implies local linear regularity. If, in addition, $Y$ is a hyperplane, then local linear regularity is equivalent to the unilateral BAP-EB.
\end{proposition}

 \begin{proof}
Assume that the unilateral BAP error bound condition from \cref{def:BAP-error-bound} holds, that is, for some point $x^*\in \bap_Y(X)$ there exist  $\omega > 0$ and   $\delta > 0$ such that, \text{ for all } $x\in \BB_{\delta}(x^{*})\cap X$,
\[
  \omega \dist(x,\bap_{Y}(X))  \leq \dist(x, \supH_{Y}(X)).
\]  
Note that, for all $x\in X$, we have $\dist(x, \supH_{Y}(X)) \leq \dist(x, Y+d)$. From the fact that $\bap_Y(X) = X\cap (Y+d)$, it follows that
 \[
 \dist(x,X\cap (Y+d)) \leq \frac{1}{\omega}\dist(x, Y+d),
\]
 for all $x\in \BB_{\delta}(x^{*})\cap X$.
Using \cite[Lemma 4.1]{Bauschke:1993} in the last inequality, with $X$ playing the role of $M$ and $Y+d$ playing the role of $N$, we get 
\[
\dist(z,X\cap (Y+d)) \leq \left(\frac{2}{\omega} + 1\right)\max\{\dist(z, X), \dist(z, Y+d) \},
\]
for all $z\in \BB_{\frac{\delta}{2}}(x^{*})$. Thus, the local linear regularity holds with $\kappa \coloneqq \frac{2}{\omega} + 1$ and  $\rho \coloneqq \frac{\delta}{2}$. 

To complete the proof, we just have to show that, if $Y$ is a  hyperplane, local linear regularity implies~\cref{eq:error-bound}, as in this case  $Y+d$ coincides with $\supH_{Y}(X)$, the optimal supporting hyperplane to $X$ regarding $Y$. Therefore, for all points  $z\in X$, $\max\{\dist(z, X), \dist(z, Y+d) \} = \dist(z, Y+d) =  \dist(z, \supH_{Y}(X))$. Hence, the result follows, with $\omega \coloneqq \frac{1}{\kappa}$ and $\delta \coloneqq \rho$.
 \end{proof}


Later, in  \cref{ex:lackoferrorbound}, we will see that if none of the sets is a hyperplane,  local linear regularity might not coincide with the unilateral BAP error bound from \cref{def:BAP-error-bound}.  

Next, we introduce a bilateral error bound condition (referred to as BAP-EB), which is more general than the unilateral BAP-EB from \cref{def:BAP-error-bound} and also implies finite convergence of MAP under inconsistency; see \cref{thm:FiniteMAP-Convex-BAPII}. Furthermore, we will see in \cref{cor:FiniteMAP-Convex-BAPII} that, under inconsistency, a given  MAP sequence converges in a finite number of steps if, and only if, the BAP-EB is satisfied along it.

\begin{definition}[BAP error bound]\label{def:BAP-error-bound-II}
  Let $X, Y \subset \RR^{n}$ be non-intersecting  closed convex sets and assume that the distance between them is attained. We say that $X$ and $Y$ satisfy the \emph{BAP error bound (BAP-EB)}   at $x^{*}\in  \bap_Y(X)$ if there exist a bound  $\omega > 0$ and a radius $\delta > 0$ such that, 
  for all $x\in \BB_{\delta}(x^{*})\cap X$, at least one of the following inequalities holds 
  \[\label{eq:BAP-error-bound-II}
  \begin{aligned}   
    \omega \dist(P_{Y}(x),\bap_{X}(Y)) & \leq \dist(P_{Y}(x), \supH_{X}(Y)),
    \\
  \omega \dist(P_{X}P_{Y}(x),\bap_{Y}(X)) &\leq \dist(P_{X}P_{Y}(x), \supH_{Y}(X)).
\end{aligned}
\]
\end{definition}



Before starting to state theorems on finite convergence of MAP and variants, we will compare BAP-EB with intrinsic transversality. Drusvyatskiy presents intrinsic transversality in \cite[Definition 3.1]{Drusvyatskiy:2015}  and comments on its importance. The definition originally considers intersecting sets, and says that two closed sets $X, Y \subset \RR^{n}$ are intrinsically transversal at a common point if there exists an angle $\alpha \in \left( 0,\frac{\pi}{2} \right] $ such that, near this common point, any two points $x \in X \backslash Y$ and $y \in Y \backslash X$ cannot have difference $x-y$ simultaneously making an angle strictly less than $\alpha$ with both the normal cones $-\Ncone_{X}(x)$ and $\Ncone_{Y}(y)$.

The concept of intrinsic transversality was adapted for nonintersecting sets in \cite{Bui:2021}. This nice manuscript  first appeared in the form of a preprint on arXiv a few days after the submission of the present paper.  The authors of \cite{Bui:2021} have similar results to ours as they derived finite convergence of MAP under inconsistency and intrinsic transversality. We will actually prove that intrinsic transversality for nonintersecting convex sets implies BAP-EB. However, the converse is not true. \cref{ex:BAP-EB-Intrinsic} shows that BAP-EB is more general. More than that, BAP-EB along a MAP sequence is a necessary and sufficient condition for its finite convergence under infeasibility; see \cref{cor:FiniteMAP-Convex-BAPII}. 
Next, we provide the definition of intrinsic transversality for two nonintersecting closed convex sets introduced in \cite[Condition 1']{Bui:2021}.

\begin{definition}[Intrinsic transversality for disjoint sets] \label{def:intr-trans}
Two closed convex sets $X, Y \subset \RR^{n}$ with empty intersection and attainable distance are \emph{intrinsically transversal} at $(x^{*},y^{*})\in\bap(X,Y)$ if there exist $\kappa\in(0,1)$ and $\delta>0$ such that  
\[\label{eq:intr-trans}
  \max \left\{\dist\left(\frac{x-y}{\norm{x-y}}, \Ncone_{Y}(y)\right), \dist\left(\frac{x-y}{\norm{x-y}},-\Ncone_{X}(x)\right)\right\} \geq \kappa,
\]
for all $x \in X \cap\BB_\delta(x^{*})$ and $y \in \big(Y \backslash \bap_{X}(Y)\big)\cap \BB_\delta(y^{*})$.  
\end{definition}

We point out that the previous definition is not symmetrical. It is unilateral, since one is considering local points $y\in Y$ such that $\dist(X,y)>\dist(X,Y)$, but this condition is not required for  points in $X$.

We now establish that intrinsic transversality is a sufficient condition for BAP-EB to be fulfilled.

\begin{proposition}[Intrinsic transversality implies BAP-EB]\label{PropIntTrans-EB2}
If intrinsic transversality from \cref{def:intr-trans} is satisfied, then BAP-EB from \cref{def:BAP-error-bound-II} holds. 
  \end{proposition}
  
\begin{proof}
 We are going to prove the statement by  showing that the absence of BAP-EB prevents intrinsic transversality to hold.  

  Consider two closed and convex sets  $X,Y\in \RR^n$ with empty intersection and attainable distance. Let $(x^{*},y^{*})\in \bap(X,Y)$ and suppose that BAP-EB does not hold at $x^{*}$. The lack of BAP-EB guarantees the existence of a sequence $(x^k)_{k\in\NN}\subset X$ converging to $x^*$ such that the sequence defined by $y^{k} \coloneqq P_{Y}(x^{k})$  has no term in $\bap_{X}(Y)$ and  converges to $y^{*}$, and the sequence determined by  $w^k \coloneqq P_{X} P_{Y}(x^{k})$ does not have any term in   $ \bap_Y(X)$ and converges to $x^{*}$. The existence of such sequence $(x^k)_{k\in\NN}$ is indeed easy to verify when denying both inequalities \cref{eq:BAP-error-bound-II} in BAP-EB.

  Obviously, $x^{k} - y^{k}\in \Ncone_Y(y^{k})$ and   $w^{k} - y^{k} \in -\Ncone_X(w^{k})$.   Taking into account the nonexpansiveness of projections, we have $x^{k} - y^{k} \to x^{*} - y^{*}$ and $w^{k} - y^{k} \to x^{*} - y^{*}$.

 Hence,
  \begin{align}
  \lim_{k\to \infty}\dist\left(\frac{w^{k}-y^{k}}{\norm{w^{k}-y^{k}}}, \Ncone_{Y}(y^{k})\right) & \leq  \lim_{k\to \infty} \norm{\frac{w^{k}-y^{k}}{\norm{w^{k}-y^{k}}} -\frac{x^{k}-y^{k}}{\norm{x^{k}-y^{k}}} } \\
   & = \norm{\frac{x^{*}-y^{*}}{\norm{x^{*}-y^{*}}} -\frac{x^{*}-y^{*}}{\norm{x^{*}-y^{*}}} } = 0,
\end{align}
and
\begin{align}
  \lim_{k\to \infty}\dist\left(\frac{w^{k}-y^{k}}{\norm{w^{k}-y^{k}}}, - \Ncone_{X}(w^{k})\right) & \leq  \lim_{k\to \infty} \norm{\frac{w^{k}-y^{k}}{\norm{w^{k}-y^{k}}} -\frac{w^{k}-y^{k}}{\norm{w^{k}-y^{k}}} } = 0, 
\end{align}
which invalids  \cref{eq:intr-trans}, \emph{i.e.}, intrinsic transversality fails. 
\end{proof}

\section{Finite convergence of projection-based methods}\label{sec:finitemethods}
\subsection{Finite convergence of MAP}\label{sec:finiteconvmap}

In this subsection we present straightforward consequences of \cref{thm:SingleStepMAP-PolyHyp,thm:SingleStep-ConvexHyperplane,thm:SingleStep-Convex} on MAP. 
The first theorem establishes finite convergence of MAP for a polyhedron and a disjoint hyperplane. We then show that under local linear regularity, MAP for a closed convex set and a disjoint hyperplane also converges in a finite number of steps, as for a hyperplane versus a closed convex set, local linear regularity coincides with BAP-EB; recall \cref{prop:EB-linear-regularity}. Then, we see in \cref{thm:FiniteMAP-Convex} that infeasibility combined with unilateral BAP-EB  implies finite convergence of alternating projections for two closed convex sets. 
Finally, in \cref{thm:FiniteMAP-Convex-BAPII}, we prove the most important result of this work, namely,  that BAP-EB from \cref{def:BAP-error-bound-II} under inconsistency also yields finite convergence of MAP.

\begin{theorem}[Finite convergence of MAP for polyhedron versus hyperplane under inconsistency]\label{thm:FiniteMAP-poly}
  Consider two nonempty sets $\Omega,H\subset \RR^{n}$ such that $\Omega$ is a polyhedron, $H$ is a hyperplane and $\Omega\cap H = \emptyset$.  Let $x^{0}\in\RR^n$ be given and    $(x^{k})_{k\in\NN}$ be the MAP sequence defined by $x^{k+1} \coloneqq P_{\Omega}P_H(x^{k})$.  Then, $(x^{k})_{k\in\NN}$ converges in finitely many steps to $\bar x \in \bap_{H}(\Omega) $. 
\end{theorem} 

\begin{proof}
 \Cref{fact:polyMAP}  provides that $\dist(\Omega,H)$ is attained, as $H$ is also a polyhedron. Hence, the sequence $(x^{k})_{k\in\NN}$ converges to some $\bar x\in \bap_{H}(\Omega)$ (\Cref{conv-MAP}(i)).
Hence,  for a sufficiently large $\bar k$, $\dist(x^{\bar k},\bap_{H}(\Omega))\leq r$, where $r$ set as in   \Cref{thm:SingleStepMAP-PolyHyp}. Therefore, by  \Cref{thm:SingleStepMAP-PolyHyp}, $x^{\bar k + 1} = P_{\Omega}P_H(x^{\bar k}) = \bar x$, hence proved. 
\end{proof}

\cref{thm:FiniteMAP-poly} alone encompasses a very relevant setting featuring non-isolated solutions as it concerns  a hyperplane versus a polyhedron (see \cref{ex:MAP-LP}  for an application on linear programming).  Below, we will state results for non-intersecting convex sets  under BAP-EB conditions beyond the polyhedral-affine setting. 

\begin{theorem}[Finite convergence of MAP for a convex set versus hyperplane under inconsistency]
\label{cor:FiniteMAP-ConvexHyperplane}
Let $X, H \subset \RR^{n}$ be  a closed convex set and a hyperplane, respectively, and suppose that $X$ and  $H$ are disjoint with attainable distance. Assume that $X$ and $H$ satisfy the BAP error bound \cref{eq:error-bound} from \Cref{def:BAP-error-bound} at  $x^{*}\in \bap_H(X)$
 with bounds  $\omega > 0$ and  radius $\delta > 0$.
By setting $r \coloneqq \min\left\{\omega \dist(X,H),\frac{\delta}{2}\right\}$, we have that  for any given $x^{0}\in\RR^{n}$, the MAP sequence $(x^k)_{k\in\NN}$ defined by $x^{k+1} \coloneqq P_{X}P_{H}(x^{k})$ converges to a point $\bar x\in \bap_{H}(X)$. Moreover, if  there exists an index $\bar  k \geq 0$, such that $x^{\bar  k} \in    \BB_{r}(x^{*})$, then $x^{\bar k + 1} = \bar x$, that is, in this case, MAP converges in at most $\bar k + 1$ steps.

\end{theorem}

\begin{proof}
 It is well-known that MAP converges globally for two closed convex sets with attainable distance (see \Cref{conv-MAP}(i)) and thus, the first part of item (ii) follows. Its second part is a straightforward consequence of \Cref{thm:SingleStep-ConvexHyperplane}. 
\end{proof}

\begin{theorem}[Finite convergence of MAP under infeasibility and unilateral BAP-EB]\label{thm:FiniteMAP-Convex}
Let $X, Y \subset \RR^{n}$ be  closed convex sets such that $X\cap Y = \emptyset$ and assume that the distance between them is attained. 
Assume that $X$ and $Y$ satisfy the unilateral BAP error bound of \Cref{def:BAP-error-bound} at  $x^{*}\in \bap_Y(X)$ with bound $\omega > 0$ and radius $\delta > 0$. 
Set $r \coloneqq \min\left\{\omega \dist(X,Y),\frac{\delta}{2}\right\}$, and then we have that for any given $x^{0}\in\RR^{n}$, the MAP sequence $(x^k)_{k\in\NN}$ defined by $x^{k+1} \coloneqq P_{X}P_{Y}(x^{k})$ converges to a point $\bar x\in \bap_{Y}(X)$. Moreover, if  there exists an index $\bar k \geq 0$, such that $x^{\bar k} \in    \BB_{r}(x^{*})$, then $x^{\bar k + 1} = \bar x$, that is, in this case, MAP converges in at most $\bar k + 1$ steps.

\end{theorem}
\begin{proof}
The result  follows directly from \Cref{conv-MAP}(i)  and \Cref{thm:SingleStep-Convex}.
\end{proof}

In order to prove finite convergence of MAP under the bilateral BAP-EB from \Cref{def:BAP-error-bound-II}, we need the following auxiliary proposition.

\begin{proposition}\label{prop:BAP-conv}
  Let $X, Y \subset \RR^{n}$ be  closed convex sets such that $X\cap Y = \emptyset$, assume that the distance   between them is attained.  Consider a sequence  $(z^k)_{k\in\NN}\subset X$ converging to a point $z^{*}\in \bap_Y(X)$. Assume there exists $\omega>0$ such that, for all $k$, 
\[\label{eq:prop-BAPconv1}
\omega \dist(z^{k},\bap_{Y}(X)) \leq \dist(z^{k}, \supH_{Y}(X)).
\]  
Then, 
\[\label{eq:prop-BAPconv2}  
  \omega \dist(z,\bap_{Y}(X)) \leq \dist(z, \supH_{Y}(X))
  \]  
  holds for all $z \in \overline S \coloneqq  \closu\left( \conv \left ( (z^k)_{k\in\NN} \cup \bap_{Y}(X)  \right )  \right)$. Moreover, $\dist(\overline{S},Y) =  \dist(X,Y)$,  $\bap_{Y}(\overline{S}) = \bap_{Y}(X) $ and  $\supH_{Y}(\overline{S}) = \supH_{Y}(X) $. 
\end{proposition}

\begin{proof}
  We start proving that the convex hull $S \coloneqq  \conv \left ( (z^k)_{k\in\NN} \cup \bap_{Y}(X)  \right ) $ satisfies the announced error bound \cref{eq:prop-BAPconv2}. This will be done by induction over the sets 
  \[
    S_k \coloneqq \conv \left( \left\{ z^{0},  z^{1},  \ldots, z^{k} \right \} \cup \bap_{Y}(X)\right). \] 
Note that $S_{0} = \{z_\lambda\in X \mid z_\lambda = \lambda z^{0} + (1-\lambda)v, \; \lambda \in [0,1], v\in \bap_{Y}(X) \} $. Taking arbitrary but fixed $v\in \bap_{Y}(X)$ and $\lambda \in [0,1]$, we define accordingly $z_{\lambda}\in S_0$ and denote  by $\bar z_\lambda \coloneqq P_{\bap_{Y}(X)}(z_\lambda)$and   by $\bar z^{0} \coloneqq P_{\bap_{Y}(X)}(z^{0})$. Note that $\hat z_\lambda \coloneqq \lambda \bar z^{0} + (1-\lambda)v \in \bap_{Y}(X)$, so we get
\begin{align}   
  \omega \dist(z_{\lambda },\bap_{Y}(X)) & = \omega \norm{z_{\lambda } - \bar z_{\lambda }}  
  \\
  & \leq \omega \norm{z_{\lambda } - \hat z_{\lambda}}  
  \\
    & = \omega \norm{ \lambda z^{0} + (1-\lambda)v - (\lambda\bar z^{0} + (1-\lambda)v) }  
  \\
  & =  \omega \norm{\lambda (z^{0} - \bar z^{0})} \\  
  & =  \lambda \omega \dist\left(z^{0}, \bap_{Y}(X)  \right).  \label{eq:prop-BAPconv3}  
  \end{align}
  Define now $\tilde z_\lambda \coloneqq P_{\supH_{Y}(X)}(z_\lambda)$. Since  $\supH_{Y}(X)$ is a hyperplane we get \[\tilde z_\lambda  = 
  P_{\supH_{Y}(X)}( \lambda z^{0} + (1-\lambda)v) = \lambda  P_{\supH_{Y}(X)}( z^{0})  + (1-\lambda)P_{\supH_{Y}(X)}(v) = 
  \lambda \tilde z^{0} + (1-\lambda)v  ,\] using in the last equality that  $\tilde z^{0}\coloneqq P_{\supH_{Y}(X)}( z^{0}) $ and $v\in \bap_{Y}(X) \subset \supH_{Y}(X)$. Then,
\begin{align}   
     \dist(z_{\lambda },\supH_{Y}(X)) & =  \norm{z_{\lambda } - \tilde z_\lambda}  
    \\
    & =  \norm{(\lambda z^{0} + (1-\lambda)v) - (\lambda \tilde z^{0} + (1-\lambda)v ) }  
    \\
    & =   \norm{\lambda (z^{0} - \tilde z^{0})} \\  
    & =   \lambda \dist\left(z^{0}, \supH_{Y}(X)  \right)) . \label{eq:prop-BAPconv4}
    \end{align}
   The  hypothesis  \cref{eq:prop-BAPconv1} for $k=0$ reads as 
\[
  \omega  \dist\left(z^{0}, \bap_{Y}(X)  \right) \leq  \dist\left(z^{0}, \supH_{Y}(X)  \right).
\]
Multiplying this inequality  by $\lambda \in [0,1]$ and combining it with equalities \cref{eq:prop-BAPconv3,eq:prop-BAPconv4}, gives us the first step of the induction.
 
Now, note that \[S_{k+1} = \conv \left( \left\{ z^{0},  z^{1},  \ldots, z^{k},  z^{k+1} \right \} \cup \bap_{Y}(X)\right) = \conv (S_{k} \cup  \{z^{k+1}\}) ,\] 
our  induction hypothesis says that all $u\in S_{k}$ satisfies 
\[      \label{eq:prop-BAPconvuomega}
  \omega  \dist\left(u, \bap_{Y}(X)  \right) \leq  \dist\left(u, \supH_{Y}(X)  \right)
\]
and inequality \cref{eq:prop-BAPconv1} specialized for $k+1$ gives us
\[\label{eq:prop-BAPconvvomega}
  \omega  \dist\left(z^{k+1}, \bap_{Y}(X)  \right) \leq  \dist\left(z^{k+1}, \supH_{Y}(X)  \right).
\]

Consider an arbitrary point in $ S_{k+1} $. It  can be  written as $z_\lambda = \lambda u + (1-\lambda)z^{k+1}$, for some   $u\in S_{k}$ and  $\lambda \in [0,1] $. Let $\bar u,\bar z^{k+1}$ be the distance realizers of $u,z^{k+1}$ with respect to $\bap_{Y}(X)$, respectively and $\tilde u, \tilde z^{k+1}$ be the distance realizers of $u,z^{k+1}$ regarding $\supH_{Y}(X)$, respectively. Note that $\tilde z_{\lambda} \coloneqq P_{\supH_{Y}(X)} (z_{\lambda}) =  \lambda \tilde u + (1-\lambda)\tilde z^{k+1}$, since $\supH_{Y}(X)$ is a hyperplane. Moreover, by the same token, $u - \tilde u = \alpha d$ and  $z^{k+1} - \tilde z^{k+1} = \beta d$, for some $\alpha, \beta \geq 0$ and where $d$ is the displacement vector (pointing from $Y$ to $X$). Thus,
\begin{align}
  \dist\left(z_{\lambda}, \supH_{Y}(X)\right)  & = \norm{ z_{\lambda} - \tilde z_{\lambda} } 
    = \norm{ \lambda  (u - \tilde u) + (1-\lambda) (z^{k+1}- \tilde z^{k+1}) } \\ 
    &  = \norm{ \lambda  (\alpha d) + (1-\lambda) (\beta d) }  = \norm{ (\lambda  \alpha  + (1-\lambda) \beta) d } \\
   & = (\lambda  \alpha  + (1-\lambda) \beta) \norm{  d } = \lambda  \alpha\norm{  d }   + (1-\lambda) \beta \norm{  d } \\
   & = \lambda  \norm{  \alpha d }   + (1-\lambda)  \norm{\beta  d } = \lambda  \norm{  u - \tilde u}   + (1-\lambda)  \norm{z^{k+1} - \tilde z^{k+1} } \\
   & = \lambda \dist\left(u, \supH_{Y}(X)  \right)  + (1-\lambda) \dist\left(z^{k+1}, \supH_{Y}(X)  \right) .       \label{eq:prop-BAPconv5}
\end{align}

Now, the point $\hat z_{\lambda} \coloneqq \lambda \bar u + (1-\lambda)\bar z^{k+1}$ belongs to $\bap_{Y}(X)$ but need not coincide with $\bar z_{\lambda}$, the distance realizer of $z_{\lambda}$ to $\bap_{Y}(X)$. Nevertheless, we can write 
  \begin{align}   
    \omega\dist(z_{\lambda },\bap_{Y}(X)) & =  \omega\norm{z_{\lambda } - \bar z_\lambda}  
    \leq   \omega\norm{z_{\lambda } - \hat z_\lambda}  \\
   & =  \omega \norm{(\lambda u + (1-\lambda) z^{k+1}) - (\lambda \bar u + (1-\lambda)\bar z^{k+1} ) }  
   \\
   & \leq    \lambda \omega\norm{ u - \bar u}  + (1-\lambda) \omega \norm{z^{k+1} - \bar z^{k+1}  }   \\
   & = \lambda \omega \dist\left(u, \bap_{Y}(X)  \right)  + (1-\lambda)\omega  \dist\left(z^{k+1}, \bap_{Y}(X)  \right) \\
      & \leq \lambda \dist\left(u, \supH_{Y}(X)  \right)  + (1-\lambda) \dist\left(z^{k+1}, \supH_{Y}(X)  \right) \\
      & = \dist\left(z_{\lambda}, \supH_{Y}(X)  \right), 
   \end{align}
where the last inequality is by \cref{eq:prop-BAPconvuomega,eq:prop-BAPconvvomega} and the last  equality is due to \cref{eq:prop-BAPconv5}. Therefore, the induction argument is completed and the error bound \cref{eq:prop-BAPconv2} holds for all points in $S$. 

Next we are going to prove that the error bound \cref{eq:prop-BAPconv2}  extends to $\overline{S}$, the closure of  $S$. Take $s \in \overline{S}$. Then, there exists a sequence $(s^{\ell})_{\ell \in \NN} \subset S$ so that $s^{\ell} \to s$. Note that we  have just proven the error bound for all points in $S$, thus  all $s^{\ell}$ satisfies
\[
  \omega \dist(s^{\ell},\bap_{Y}(X)) \leq \dist(s^{\ell}, \supH_{Y}(X)).
\] 
Taking into account that the distance functions to both sets $\bap_{Y}(X)$ and $\supH_{Y}(X)$ are continuous, we can take limits as $\ell$ goes to infinity in both sides of the previous inequality, getting the result. 

Finally, since  $\bap_{Y}(X) \subset \overline{S}\subset X$,  we have $\dist(X,Y) \leq \dist(\overline{S},Y)\leq \dist(\bap_{Y}(X),Y)= \dist(X,Y)$, that is,  $\dist(\overline{S},Y) =  \dist(X,Y)$. Thus,   $ \bap_{Y}(\overline{S})  = \bap_{Y}(X)  $ and $ \supH_{Y}(\overline{S})  = \supH_{Y}(X)  $  
\end{proof} 

We are now ready to establish finite convergence of MAP under inconsistency added by BAP-EB. 

\begin{theorem}[Finite convergence of MAP under infeasibility and BAP-EB]\label{thm:FiniteMAP-Convex-BAPII}
    Let $X, Y \subset \RR^{n}$ be  closed convex sets such that $X\cap Y = \emptyset$, assume that the distance   between them is attained and let  $x^{0}\in \RR^{n}$ be given. Then, the MAP sequence $(x^k)_{k\in\NN}$,  defined by $x^{k+1} \coloneqq P_{X}P_{Y}(x^{k})$,  converges to some $x^{*}\in \bap_{Y}(X)$. 
    If  BAP-EB  from \Cref{def:BAP-error-bound-II} is satisfied by $X$ and $Y$ at $x^{*}$, then $(x^k)_{k\in\NN}$ converges to  $x^{*}$  in a finite number of steps.   
  
  \end{theorem}

\begin{proof}
  The convergence of the MAP sequence to some $x^{*} \in \bap_{Y}(X)$ for two disjoint closed convex sets $X$ and $Y$, assuming its distance is  attainable, is due to \Cref{conv-MAP}(i). Assume BAP-EB as in the statement of the theorem and that the MAP sequence $(x^k)_{k\in\NN}$ does not converge in a finite number of steps.  Thus, since MAP iterates satisfy $\norm{x^{k+1}-x^{*}}\leq \norm{x^{k}-x^{*}}$, all the terms $x^{k}$ lie outside $\bap_{Y}(X)$ and, in particular, they are all different from $x^{*}$.  

  Note that BAP-EB consists of all points $x\in X$ near $x^{*}$ fulfilling at least one of the two inequalities \cref{eq:BAP-error-bound-II} given in \cref{def:BAP-error-bound-II}. In the following we divide our proof in two cases.

  First, let us assume that there is a subsequence of the MAP sequence satisfying the second inequality in  \cref{eq:BAP-error-bound-II}. 
   So, we have an infinite set of indexes $J\subset \NN$, such that
  \[
    \omega \dist(P_{X}P_{Y}(x^{k-1}),\bap_{Y}(X)) \leq \dist(P_{X}P_{Y}(x^{k-1}), \supH_{Y}(X)),  \text{ for all } k\in J,
  \]
which is the same as 
   \[\label{eq:ThmBAP2}
     \omega \dist(x^k,\bap_{Y}(X)) \leq \dist(x^k, \supH_{Y}(X)), \text{ for all } k\in J.
  \]
 
  Let us define the closed convex set  $\overline{S}  \coloneqq \closu\left( \conv \left ((x^{k})_{k\in J} \cup \bap_{Y}(X) \right )\right) \subset X$.  Clearly, $x^{*}\in\overline{S}$ and, since \cref{eq:ThmBAP2} holds and $(x^{k})_{k\in J}$ converges to $x^{*}$, \cref{prop:BAP-conv} can be used. Then, for all $z \in \overline{S} $, 
  \[
    \omega \dist(z,\bap_{Y}(\overline{S})) = \omega \dist(z,\bap_{Y}(X)) \leq \dist(z, \supH_{Y}(X)) =  \dist(z, \supH_{Y}(\overline{S})).
\]
This inequality means that the unilateral BAP-EB condition from \cref{def:BAP-error-bound}  is satisfied for $\overline{S}$ and $Y$ at $x^*$. Therefore, \cref{thm:SingleStep-Convex} applies because, for all $k\in \NN$ sufficiently large, $x^{k-1}$ is near enough to $ x^{*}$. Hence, for all large $k$, we have $P_{\overline{S}}P_Y(x^{k-1}) \in \bap_{Y}(\overline{S})=\bap_{Y}(X)$, where the equality of the BAP sets follows from \cref{prop:BAP-conv}.

Recall that, for all $k\in J$, $x^{k}$ belongs to $\overline{S}\subset  X$. Moreover, for each  $k\in J$ there exists $j_k\in \NN$ such that $k+j_k\in J$ and thus  $x^{k+j_k}\in \overline{S}$, and, in addition, $\norm{x^{k+j_k} - x^{*}}\leq \norm{x^{k} - x^{*}}$. 

Note that $x^{k+j_k} = P_XP_Y(x^{k+j_k -1})$ and that $P_Y(x^{k+j_k -1})$ approaches  $P_Y(x^{*})$, when $k$ goes to infinity. Thus, for all large $k\in J$,  $P_{\overline{S}}P_Y(x^{k+j_k -1}) \in \bap_{Y}(X)$. Then, by the characterization of projection onto convex sets, we get 
\begin{align}
 \scal{P_Y(x^{k+j_k -1}) - P_{\overline{S}}P_Y(x^{k+j_k -1})} {x^{k+j_k} - P_{\overline{S}}P_Y(x^{k+j_k -1})}  \leq 0.
\end{align}
Note that both vectors in the above inner product are different from zero, because $P_{\overline{S}}P_Y(x^{k+j_k -1})$ is in $\bap_Y(X)$ but not in $Y$ and $x^{k+j_k}$ is in $\overline{S} $ but not in $\bap_{Y}(X)$. Then, \[\norm{P_Y(x^{k+j_k -1}) - x^{k+j_k}}>\norm{P_Y(x^{k+j_k -1}) - P_{\overline{S}}P_Y(x^{k+j_k -1})}.\] On the other hand, since $\overline{S}\subset X$, we have   
\begin{align}
\norm{P_Y(x^{k+j_k -1}) - P_{\overline{S}}P_Y(x^{k+j_k -1})} \ge & \norm{P_Y(x^{k+j_k -1}) - P_XP_Y(x^{k+j_k -1})}  \\ = & \norm{P_Y(x^{k+j_k -1}) - x^{k+j_k}}  \\ > &  \norm{P_Y(x^{k+j_k -1}) - P_{\overline{S}}P_Y(x^{k+j_k -1})}, \label{eq:ThmBAP-contr}
\end{align}
which is a contradiction.

For second case, suppose now that there exists a subsequence $(x^k)_{k\in J'}$ of the MAP sequence $(x^k)_{k\in\NN}$ satisfying the first inequality in  \cref{eq:BAP-error-bound-II}, that is,  
\[
  \omega \dist(P_{Y}(x^{k}),\bap_{X}(Y))  \leq \dist(P_{Y}(x^{k}), \supH_{X}(Y)), \text{ for all } k\in J'.
\]
By defining  the sequence $(y^{k})_{k\in J'} $, where $y^{k}\coloneqq P_{Y}(x^{k})$, we have 
\[\label{eq:ThmBAP-yk}
  \omega \dist(y^{k},\bap_{X}(Y))  \leq \dist(y^{k}, \supH_{X}(Y)), \text{ for all } k\in J'.
\]
Then, the proof can be carried out analogously to the one in the first case  by changing the roles of $X$ and $Y$.  Defining the closed convex set  $\overline{S}'  \coloneqq \closu\left( \conv \left ((y^{k})_{k\in J'} \cup \bap_{X}(Y) \right )\right) \subset Y$, we can use   \cref{prop:BAP-conv}, since $(y^{k})_{k\in J'}$ converges to $y^{*}$ and \cref{eq:ThmBAP-yk} holds, to get the unilateral BAP-EB now for $\overline{S}'$ and $X$ at $y^{*}$. Therefore, \cref{thm:SingleStep-Convex} applies  to prove that $P_{\overline{S}'}P_X(y^{k-1}) \in \bap_X(\overline{S}')=\bap_X(Y)$, for all $k$ sufficiently large. So, we also derive a contradiction similar to \cref{eq:ThmBAP-contr}.

Thus, the MAP sequence $(x^k)_{k\in\NN}$ converges to $x^{*}$  in a finite number of steps.
\end{proof}

We finalize this subsection with a necessary and sufficient condition for a MAP sequence to converge in a finite number of steps.  

\begin{theorem}[Finite convergence of a MAP sequence under infeasibility and BAP-EB]\label{cor:FiniteMAP-Convex-BAPII}
  Let $X, Y \subset \RR^{n}$ be  closed convex sets such that $X\cap Y = \emptyset$, assume that the distance   between them is attained. Let  $x^{0}\in \RR^{n}$ be given and consider the MAP sequence $(x^k)_{k\in\NN}$,  defined by $x^{k+1} \coloneqq P_{X}P_{Y}(x^{k})$ with limit point $x^{*}\in \bap_{Y}(X)$. The sequence  $(x^k)_{k\in\NN}$ converges in a finite number of steps if, and only if, there exists $\omega >0 $ such that, for all $k\in\NN\backslash\{0\}$, at least one of the two inequalities holds
  \[\label{eq:cor-BAP-error-bound-II}
  \begin{aligned}   
    \omega \dist(P_{Y}(x^{k}),\bap_{X}(Y)) & \leq \dist(P_{Y}(x^{k}), \supH_{X}(Y)),
    \\
  \omega \dist(x^{k+1},\bap_{Y}(X)) &\leq \dist(x^{k+1}, \supH_{Y}(X)).
\end{aligned}
\]  
\end{theorem}

\begin{proof}
  Assume that the MAP sequence $(x^k)_{k\in\NN}$   converges in a finite number of steps to $x^{*}$ and reaches this limit point for the first time at iteration $\bar k \in\NN\backslash\{0\}$. Thus, for any positive $\omega$ both inequalities in \cref{eq:cor-BAP-error-bound-II} are trivially satisfied for all $k \geq \bar k$. That said, we are done if $\bar k = 1$.
  Suppose $\bar k \geq   2$. Then, the positive $\omega$ defined by 
  \[\omega \coloneqq \min_{0\leq k < \bar k - 1}\left\{ \frac{\dist(x^{k+1}, \supH_{Y}(X))}{ \dist(x^{k+1},\bap_{Y}(X))}\right\} \]
yields the result, since for all $1\leq k<\bar k$, we have $x^{k}\notin \bap_{Y}(X)$ and $x^{k}\notin \supH_{Y}(X)$.

Reciprocally, assume the existence of $\omega>0$ as in the statement of the corollary. Then, we can follow the exact same lines of the proof of \cref{thm:FiniteMAP-Convex-BAPII} just replacing the sequence at the beginning of that proof by the MAP sequence $(x^k)_{k\in\NN}$ and the result follows. 
\end{proof}

\subsection{Finite termination of other projection methods}\label{sec:finite-extensions}

In this subsection, we present extensions of some results of \Cref{sec:finiteconvmap}. We address two families of convex feasibility problems. The first deals with two closed convex sets with empty intersection, while the second concerns finite number of closed convex sets having no point in common.

\subsubsection{Best-pair tracking}

Here we discuss  how
our results 
apply to a broader class of methods tracking best approximation pairs. Interestingly, one may have a method generating a sequence of infinitely many distinct points for which a MAP step intrinsically provides a finite termination criterion. This is formally presented in the next theorem which, in turn, will serve to identify finite termination of the well-known Cimmino method and the famous Douglas-Rachford method.  

For this  subsection, let $X, Y \subset \RR^{n}$ be  closed convex sets such that $X\cap Y = \emptyset$ and assume that the distance between $X$ and $Y$ is attained. 


\begin{lemma}[Intrinsic finite termination]\label{thm:intrisic-finite} For a given $z^{0}\in\RR^{n}$, consider a method generating a sequence $(z^k)_{k\in\NN}$ such that the shadow sequence onto $X$, $(P_X(z^k))_{k\in\NN}$, converges to a point $\hat x \in \bap_{Y}(X)$. If the unilateral BAP-EB  from \cref{def:BAP-error-bound} is satisfied at $\hat x$, then there exists an index $\hat k \geq 0$, such that for all $k\geq \hat k$, we have $P_{X}P_{Y}(P_{X}(z^{k})) = P_{X}P_{Y}(P_{X}(z^{\hat k}))\in \bap_{Y}(X)$.
\end{lemma}

\begin{proof}
The result is a consequence of \Cref{thm:SingleStep-Convex}. Indeed, let $\hat x$ play the role of $x^*$ in \Cref{thm:SingleStep-Convex}. Since the shadow sequence $(P_X(z^k))_{k\in\NN}$ converges to $\hat x$, there exists $\hat k$ such that $P_X(z^k)\in \BB_{r}(\hat x)$ for all $k\ge \hat k$ where $r = \min\left\{\omega \dist(X,Y),\frac{\delta}{2}\right\}$. Hence, \Cref{thm:SingleStep-Convex} can be employed and the result follows. 
\end{proof}





We proceed by  enforcing \Cref{thm:intrisic-finite} for the Cimmino method~\cite{Cimmino:1938}. 
\begin{corollary}[Cimmino’s intrinsic finite termination]\label{corol:Cimmino-finite}
For a given $z^{0}\in\RR^{n}$, consider the Cimmino sequence $(z^k)_{k\in\NN}$ given by $z^{k+1} \coloneqq \frac{1}{2}(P_{X}(z^{k})+P_{Y}(z^{k}))$. Then, there exists a best approximation pair $(\bar x, \bar y)\in \bap(X,Y)$  such that $z_k \to \frac{1}{2}(\bar x + \bar y)$. If, in addition,  the unilateral BAP-EB  from \cref{def:BAP-error-bound} is satisfied at $\bar x$, then there exists an index $\hat k \geq 0$, such that for all $k\geq \hat k$, we have $P_{X}P_{Y}(P_{X}(z^{k})) = P_{X}P_{Y}(P_{X}(z^{\hat k}))\in \bap_{Y}(X)$.
\end{corollary}

\begin{proof} The convergence of the sequence $(z^k)_{k\in\NN}$ to  the midpoint of a best approximation pair can be found in \cite[Theorem 6.3]{Bauschke:1994}. Then, of course, both  shadow sequences $(P_X(z^k))_{k\in\NN}$ and $(P_Y(z^k))_{k\in\NN}$ converge to best points $\bar x$ and $\bar y$, respectively. Thus, the hypotheses of \Cref{thm:intrisic-finite} are fulfilled, providing the corollary. 
\end{proof}

The last result of this section shows that we can as well suit \Cref{thm:intrisic-finite} for the Douglas-Rachford method~\cite{Douglas:1956}. Although this method   always diverges under inconsistency~\cite[Theorem 3.13(ii)]{Bauschke:2004a}, one of its shadows detects a best point. This enables us to apply  \Cref{thm:intrisic-finite}.

\begin{corollary}[Douglas-Rachford’s intrinsic finite termination]\label{corol:DR-finite}
For a given $z^{0}\in\RR^{n}$, consider the Douglas-Rachford sequence $(z^k)_{k\in\NN}$ given by $z^{k+1} \coloneqq \frac{1}{2}(z^k + R_{Y}R_{X}(z^{k}))$, where $R_X\coloneqq 2P_X-\Id$ and $R_Y\coloneqq 2P_Y-\Id$ are the reflectors through $X$ and $Y$, respectively. Then, there exists  $\bar x\in \bap_Y(X)$  such that the shadow sequence   $(P_X(z^k))_{k\in\NN}$ converges to $\bar x $. If, in addition,  the unilateral BAP-EB  from \cref{def:BAP-error-bound} is satisfied at $\bar x$, then there exists an index $\hat k \geq 0$, such that for all $k\geq \hat k$, we have $P_{X}P_{Y}(P_{X}(z^{k})) = P_{X}P_{Y}(P_{X}(z^{\hat k}))\in \bap_{Y}(X)$.
\end{corollary}

\begin{proof} The convergence of the shadow sequence $(P_X(z^k))_{k\in\NN}$ to a point $\bar x \in \bap_Y(X)$ is ensured by 
\cite[Theorem 3.13 and Remark 3.14(ii)]{Bauschke:2004a}. So, \cref{thm:intrisic-finite} applies and the result follows. 
\end{proof}

Finally, it is worth commenting that the famous Dykstra's method is also contemplated by \Cref{thm:intrisic-finite}, as it generates a sequence converging to a best approximation pair; see \cite[Theorem 3.7]{Bauschke:1994}.

\subsubsection{Inconsistent multi-set intersection}
In this section, we consider a finite number of closed convex sets $X_1,X_2,\ldots, X_m \subset \RR^n$.  If these sets have nonempty intersection, a point common to them is found, for instance, by the method of Cyclic projections~\cite[Corollary 5.26]{Bauschke:2017} (which is an extension of MAP) or by the Cimmino method~\cite{Cimmino:1938} applied to multi-set intersection problems (also called simultaneous projection  method or method of barycenters); for further discussions see \cite{Censor:1988a,Censor:2014}.
When $X_1\cap X_2\cap \cdots\cap  X_m = \emptyset$ and one of the target sets is bounded,   the method of Cyclic projections converge to a point that provides what we call a {cycle} (see \cite[Theorem 2]{Gubin:1967},\cite[Theorem 5.4.1]{Bauschke:1997},\cite[Corollary 5.24]{Bauschke:2017}) and  the Cimmino method converges to a point minimizing the sum of the squares of the distances to the target sets~\cite[Theorem 4]{Combettes:1994}.

 A point  $\bar y \in \RR^n$ is said to provide a \emph{cycle} with respect to the index order $\{1,2,\ldots,m\}$ if $P_{X_1}(\bar y)$ is a fixed point of the
 Cyclic projection  operator  $P_{X_1}P_{X_{2}}\cdots P_{X_m} $. The correspondent cycle is the tuple  $(\bar y_{1},\bar y_{2},\ldots, \bar y_{m})\in X_1\times X_2\times \cdots\times  X_m$, such that $\bar y_1 = P_{X_{1}}(\bar y)$ and 
\[\label{eq:cycles}
\bar{y}_{1}=P_{X_{1}} (\bar{y}_{2}), \; \ldots\;, \; \bar{y}_{m-1}=P_{X_{m-1}} (\bar{y}_{m}), \; \bar{y}_{m}=P_{X_{m}}( \bar{y}_{1}).
\]
Note that in the two-set case, cycles reduce to best approximation pairs.  Also, any best approximation pair is a cycle in both possible index orders. 

Our aim is to establish sufficient conditions for finite convergence of Cyclic projections and the Cimmino method when the underlying sets have empty intersection. This is done in two statements, both connected to \cref{thm:SingleStep-Convex,thm:FiniteMAP-Convex}. In the first one, we look at Cyclic projections. The second theorem explores the bond between the Cimmino method for the multi-set intersection problem and MAP applied to Pierra’s product space reformulation.

 It is worth reemphasizing that a Cimmino limit point has the nice variational characterization of minimizing the sum of the squares of the distances to the target sets.  We point out that cycles coming from Cyclic projections do not have a variational characterization; this was proven by Baillon, Combettes, and Cominetti in \cite{Baillon:2012}.  
Recent advances in the study of Cyclic projections were derived in \cite{Alwadani:2020}.

Next, we present a result tracking limit points of  Cyclic projections in a finite number of iterations.

\begin{corollary}[Finite convergence of Cyclic projections]\label{thm:Finite-Cyclic}
Let $X_1,X_2,\ldots, X_m \subset \RR^n$ be closed convex sets with empty intersection. Assume also that a cycle exists, that is, the operator $P_{X_1}P_{X_2}\cdots P_{X_m}$ has a fixed point.  Let $x^{0}\in\RR^{n}$ be given and consider  the Cyclic projection sequence $(x^k)_{k\in\NN}$ defined by $$x^{k+1} \coloneqq P_{X_1}P_{X_2}\cdots P_{X_m}(x^{k}).$$

\begin{listi}

\item Then,  $(x^k)_{k\in\NN}$   converges to a point $\bar y_1\in X_1$, which provides the cycle $(\bar y_1,\bar y_2,\ldots, \bar y_m)$, where $
\bar{y}_{1}=P_{X_{1}} (\bar{y}_{2}), \; \ldots\;, \; \bar{y}_{m-1}=P_{X_{m-1}} (\bar{y}_{m}), \; \bar{y}_{m}=P_{X_{m}}( \bar{y}_{1})$.


\item Suppose there exists an index $\bar \imath \in\{1,\ldots,m\}$ such that the cycle points $\bar y_{\bar \imath}$ and $\bar y_{\bar \imath+1}$ form a best approximation pair to $X_{\bar \imath}$ and $X_{\bar \imath + 1}$, that is, $(\bar y_{\bar \imath},\bar y_{\bar \imath+1})\in \bap(X_{\bar \imath},X_{\bar \imath + 1})$ and assume that $X_{\bar \imath}$ and $X_{\bar \imath + 1}$ satisfy the BAP error bound \cref{eq:error-bound} from \Cref{def:BAP-error-bound}  at  $\bar y_{\bar \imath} \in \bap_{X_{\bar \imath + 1}}(X_{\bar \imath })$, that is,  
\[\label{eq:error-bound-cond-Cyclic}
  \omega \dist(x,\bap_{X_{\bar \imath + 1}}(X_{\bar \imath })) \leq \dist(x, \supH_{X_{\bar \imath + 1}}(X_{\bar \imath })), \text{ for all } y\in \BB_{\delta}(\bar y_{\bar \imath})\cap X_{\bar \imath },
\]
with bound  $\omega > 0$ and  radius $\delta > 0$, and  if  $\bar \imath  = m$,  $\bar y_{\bar \imath +1} = \bar y_{1}$. 
Then, we have the existence of an index $\bar k   \geq 0$, such that $x^{ k } = \bar y_1$ for all $k\geq \bar k$, that is, in this case, the Cyclic projection sequence $(x^k)_{k\in\NN}$  converges in at most $\bar k $ steps.


\end{listi}
\end{corollary}

\begin{proof}

  The convergence  of the method of Cyclic projections, under the existence of a cycle, to a point $\bar y_1\in X_1$ that lies in   $\Fix P_{X_1}P_{X_2}\cdots P_{X_m}$  can be found in \cite{Gubin:1967}.  It is straightforward to see that $
  \bar{y}_{1}=P_{X_{1}} (\bar{y}_{2}), \; \ldots\;, \; \bar{y}_{m-1}=P_{X_{m-1}} (\bar{y}_{m}), \; \bar{y}_{m}=P_{X_{m}}( \bar{y}_{1})$. Thus, item (i) follows.

  Now, the proof of  item (ii) relies entirely on \Cref{thm:SingleStep-Convex}, except for  a few details we are going to check.  Assume, without loss of generality, that $\bar \imath = 1$.
  
  We will show  the existence of a radius $r>0$ such that for all $z \in  \BB_{r}(\bar y_{2})\cap X_{2}$ we have  $P_{X_{1}}(z) = \bar y_{1}$, which suffices to establish item (ii). This will indeed complete the proof, because $(x^k)_{k\in\NN}$ converging to $\bar y_1$ implies that the sequence $(P_{X_2}\cdots P_{X_m}(x^{k}))_{k\in\NN}$ converges to $\bar y_2$, minding the continuity of projections. Hence, there exists an index $\bar k\geq 0$ such that $P_{X_2}\cdots P_{X_m}(x^{\bar k - 1}) \in\BB_{r}(\bar y_2)$. Thus, \[x^{\bar  k} = P_{X_1}P_{X_2}\cdots P_{X_m}(x^{\bar k - 1}) = \bar y_1\] and then for all $k\geq \bar k$, we have  $x^k = \bar y_1 $ because  $ \bar y_1 $ is a fixed point of the Cyclic projection operator $P_{X_1}P_{X_2}\cdots P_{X_m}$. 

\Cref{thm:SingleStep-Convex} guarantees the existence of a radius $r>0$ such that for all $ v\in \BB_r$ $(\bar y_{1})$, $P_{X_1} P_{\supH_{X_{1}}(X_{2})}(v) \in  \bap_{\supH_{X_{1}}(X_{2})}(X_{1}) = \bap_{X_{2}}(X_{1}) $. Take an arbitrary but fixed $v \in \BB_r(\bar y_{1}) $ and define $z = P_{X_{2}}(v)$. By the nonexpansiveness of projections, $z\in \BB_r(\bar y_{2})$. Now when projecting this point $z$ onto $X_{1}$ we cross the optimal supporting hyperplane $\supH_{X_{1}}(X_{2})$ at $z^{\diamond}$. 

Note that $ z^{\diamond} + d$ belongs to $\BB_r(\bar y_{1})$, where $d = \bar y_{1} - \bar y_{2}$ is the displacement vector. Furthermore, $ P_{X_1} (z) = P_{X_1} (z^{\diamond}) = P_{X_1}  P_{\supH_{X_{1}}(X_{2})}( z^{\diamond} + d)\in \bap_{\supH_{X_{1}}(X_{2})}(X_{1}) = \bap_{X_{2}}(X_{1})$. Therefore, $P_{X_1} (z)\in \bap_{X_{2}}(X_{1})$ must be $\bar y_{1}$ because of the convergence of the sequence $(x_{k})_{k\in \NN}$ to $\bar y_{1}$.
\end{proof}

We have just investigated the obtainment of a cycle by Cyclic projections in a finite number of iterations. It is easy to see that our hypotheses are sufficient but not necessary. Indeed, if one takes three non-intersecting parallel lines in $\RR^2$, the method of Cyclic projections always captures a cycle in a finite number of iterations, although some hypotheses of \cref{thm:Finite-Cyclic} are not fulfilled, including the error bound condition \cref{eq:error-bound-cond-Cyclic}.
Note that this  error bound  holds, for instance, when  the normal cone regarding one of the targets sets, say $X_{i}$, at a respective cycle point $\bar y_{i}$ contains a neighborhood  of the previous cycle point $\bar y_{i + 1}\in X_{i +1}$. However, the result in \cref{thm:Finite-Cyclic} has to be taken with moderate enthusiasm, as even for a very simple problem with plenty of structure, Cyclic projections may not converge finitely when we have more than two sets. Consider, for instance, three segments as target sets forming an equilateral triangle in $\RR^2$ and fix an order for the sets. The associated cycle is not archived by cyclic projections in finite number of steps for every point near the cycle.

 We now move towards the investigation of finite convergence of the Cimmino method. The Cimmino iteration regarding the sets $X_1,X_2,\ldots, X_m$ computed at a point $z^k\in \RR^n$ is given by $C(z^{k})\coloneqq  \frac{1}{m}\left(P_{X_1}+ P_{X_2}+ \cdots + P_{X_m}\right)(z^k)$. We remark that this iteration does not depend on any index order, as it works in parallel. Moreover, it can be seen as a MAP step under Pierra’s product space reformulation~\cite{Pierra:1984}. Pierra’s approach considers the Cartesian product  $\XX\coloneqq X_1\times X_2\times \cdots\times  X_m $ and the diagonal subspace $\DDiag\coloneqq\{(z,\ldots,z)\in \RR^{mn}\mid z \in \RR^n\}$. Clearly, the intersection of $X_1,X_2,\ldots, X_m$ is related to the intersection of $\XX$ and $ \DDiag$, that is, $z\in \bigcap_{i=1}^m X_i$  if, and only if, the diagonal point $(z,\ldots,z)\in \DDiag$ lies in $\XX$.

It is well-known  that, for any $z^k\in \RR^n$, 
\[
P_{\DDiag}P_{\XX}(z^k,z^k,\ldots,z^k) = (C(z^k),C(z^k),\ldots,C(z^k));
\] see \cite[Theorem 1.1]{Pierra:1984}.
Hence,   Cimmino iterations can indeed always be regarded as MAP steps. This allows us to connect \Cref{thm:SingleStep-Convex} with the sought of  a Cimmino limit point in a finite number of iterations. Formally, we have the following result. 

\begin{corollary}[Finite termination of multi-set Cimmino (MAP in product space)]\label{thm:FiniteMAP-Product} Let $X_1,X_2,\ldots, X_m \subset \RR^n$ be closed convex sets with no point in common,  assume that at least one the sets is bounded and define  \(\DDiag\coloneqq\{(z,\ldots,z)\in \RR^{mn}\mid z \in \RR^n\}\text{ and } \XX\coloneqq X_1\times X_2\times \cdots\times  X_m \). For a given ${x}^0\in\RR^{n}$, set $\mathbf{x}^0\coloneqq({x}^0,\ldots,{x}^0)\in \DDiag$ and consider  the MAP sequence $(\mathbf{x}^{k})_{k\in\NN}$ such that $\mathbf{x}^{k+1}\coloneqq P_{\DDiag}P_{\XX}(\mathbf{x}^k)$.
\begin{listi}

\item Then, the MAP iterates  $\mathbf{x}^{k} = (x^k,\ldots,x^k)$  converge to a point  $\hat{\mathbf{x}} \coloneqq (\hat x,\ldots,\hat x) \in\bap_{\XX}(\DDiag)$, with $\hat x\in \RR^n$ being the limit point of the Cimmino sequence  $(x^k)_{k\in \NN}$ ruled by $x^{k+1} \coloneqq  C(x^{k}) = \frac{1}{m}\sum_{i=1}^{m}P_{X_i}(x^k)$. 

\item The Cimmino limit point $\hat x$ is a solution of the least squares problem of minimizing the sum of the squares of the distances to the target sets
$X_1,\ldots,X_m$. 


\item If, in addition, there exists  an error bound $\omega > 0$ and a radius $\delta > 0$ such that
 \[\label{eq:error-bound-cond-Prod}
 \omega \dist(x,\bap_{\XX}(\DDiag)) \leq \dist(x, \supH_{\XX}(\DDiag)),
 \]
for all $\mathbf{x}\in \BB_{\delta}(\hat{\mathbf{x}} )\cap \DDiag$, 
 where $\supH_{\XX}(\DDiag)$ is the optimal supporting hyperplane to $\DDiag$ regarding $\XX$ given in \Cref{def:optimalhyperplane},  then there exists an index $\bar k \geq 0$, such that $\mathbf{x}^{k}  = \hat{\mathbf{x}} $, for all $k\geq \bar k$.

\end{listi}

\end{corollary}

\begin{proof} Since there is  no common point to the target sets $X_1,\ldots,X_m$, it follows that $\DDiag\cap \XX = \emptyset$. 
We note that, by \cite[Proposition 4]{Combettes:1994}, 
$\bap_{\XX}(\DDiag) = \Fix P_{\DDiag}P_{\XX} = \{(z,\ldots,z)\in \RR^{mn}\mid z \in \Fix C\}$, where $\Fix C$ is the set of fixed points of the Cimmino operator $C$. Now, since one of the sets $X_i$ is bounded,  \cite[Proposition 7]{Combettes:1994} guarantees  that $\bap_{\XX}(\DDiag)\neq \emptyset$. In particular, the distance between $\DDiag$ and $\XX$ is attained. 
That said, items (i) and (ii) are completely covered by \cite[Theorems~1 and~4]{Combettes:1994}
and then item (iii) follows as a direct application of \cref{thm:SingleStep-Convex}, with $\DDiag$ playing the role of $X$ and $\XX$ playing the role of $Y$. 
\end{proof}

One has to humble their expectations as for the range of instances covered by \Cref{thm:FiniteMAP-Product}, since BAP-EB in the original space may not be transported to the product space. We present, in \Cref{ex:Cimmino-MAP-ProductSpace}, one problem where BAP-EB is transported to the product space and one where it is not.

\section{Examples and applications discussing our results}\label{sec:discussion}

This section illustrates our results upon insightful examples, one of which raises a conjecture. 

We start connecting Linear Programming with  \Cref{thm:SingleStepMAP-PolyHyp,thm:FiniteMAP-poly}. Second, we present an example showing that local linear regularity (subtransversality) combined with infeasibility is not sufficient for finite convergence of MAP.  \cref{ex:Cimmino-MAP-ProductSpace} contains an important remark on  BAP-EB in Pierra's product space.  We then  present  an example in which we raise a conjecture regarding the substitution of the Lipschitzian error bound of BAP-EB by a H\"older one. Finally, we look at an application of MAP for non-smooth convex minimization given by min-max problems. 

\begin{example}[MAP for LP]\label{ex:MAP-LP}
A curious geometrical interpretation of Linear Programming (LP) follows from \Cref{thm:SingleStepMAP-PolyHyp,thm:FiniteMAP-poly}. Consider the following linear program
\begin{equation}\label{LP}
\begin{aligned}
& {\text{minimize}}
& & \scal{c}{x} \\
& \text{subject to}
& & \Omega:=\{x\in \RR^n \mid Ax\leq b,\mbox{ } x\geq 0\} 
\end{aligned}
\end{equation} 
where $A\in \RR^{m\times n}$, $c\in \RR^n$ with $c\neq 0$ and $b\in \RR^m$ are given. Let us assume that \cref{LP} is solvable and denote its solution set by $\Omega^*$. These hypotheses guarantee the existence of a dual feasible point $\hat y\in \RR^m$, \emph{i.e.}, $A^\top \hat y\geq c$ and $\hat y\geq 0$.
Take any $\varepsilon>0$, some $\hat x \in \RR^m$ such that $\scal{c}{\hat x}=\scal{b}{\hat y}-\varepsilon $, and define $H:=\{x\in \RR^n \mid \scal{c}{\hat x-x} = 0\}$. Note that finding such $\hat x$ is a trivial task. Now, strong duality yields $\Omega\cap H=\emptyset$, so \Cref{thm:FiniteMAP-poly} applies. This means that for any starting point $x^0\in\RR^n$ the MAP sequence $x^{k+1}:=P_\Omega P_H(x^k)$ finds a best point $\bar x\in\Omega$ after a finite number of iterations. It is easy to see that the displacement vector $d=\bar x-P_H(\bar x)$ between $H$ and $\Omega$ is a positive multiple of the cost $c$. Therefore, $\bar x$ solves the LP \cref{LP}. We could indeed find $\bar x$ in a single MAP step by   pushing the hyperplane $H$ sufficiently away from $\Omega$ in the direction of $-c$.
A similar conclusion is derived in \cite{Nurminski:2016} for LP using a projection-based scheme, nevertheless  uniqueness of primal and dual solutions of \cref{LP} is required. In turn, our results hold regardless of solutions being unique.

\end{example}


 We now present an example showing that, in general, inconsistency combined with local linear regularity  (subtransversality)  from \cref{def:linear-regularity} alone does not imply finite convergence of MAP. The example features two disjoint reverse lines forming forty-five degrees. In this example, BAP-EB does not hold and thus, neither does intrinsic transversality.

\begin{example}[Linear regularity does not imply finite convergence of MAP]\label{ex:lackoferrorbound}
Let \[X\coloneqq\{(t,0,0)\in\RR^3 \mid t\in \RR\}\quad \mbox{and} \quad Y_\gamma\coloneqq \{(t,t,\gamma)\in\RR^3 \mid t\in\RR\},\] with $\gamma\in\RR$ fixed.  In this example, the unique best approximation pair consists of $(\bar x,\bar y)$,  where $\bar x=(0,0,0)$ and $\bar y=(0,0,\gamma)$, but unilateral BAP-EB from \Cref{def:BAP-error-bound} is not satisfied, for $\gamma\neq 0$. Indeed,  in this case, the optimal supporting hyperplane to $X$ regarding $Y_\gamma$ is given by $\supH_{Y}(X)=\{(x_1,x_2,x_3)\in \RR^3 \mid x_3=0\}$ and obviously, $X\cap Y_\gamma=\emptyset$. Here,
 $\dist((0,1,0), \bap_{Y_\gamma}(X))=1$, however $\dist((0,1,0), \supH_{Y}(X))=0$, which prohibits the error bound inequality \cref{eq:error-bound} to hold if $\gamma\neq 0$. 
Let us now investigate the behavior of MAP in this case.  Consider an arbitrary starting point $x^0=(a,b,c)\in \RR^3$ and the MAP sequence $x^{k+1}=P_XP_{Y_\gamma}(x^k)$. Straightforward manipulations give us \[x^k=\left(\frac{a}{2^{k}},0,0\right).\] This sequence converges linearly with rate $1/2$ to the unique best point to $X$, namely the origin $(0,0,0)$. Note that the sequence does not depend on the parameter $\gamma$, showing that, in this example, infeasibility does not play a role at all. Thus, infeasibility together with polyhedrality is not enough to guarantee finite convergence of MAP.  This example also shows that the hyperplane in both \cref{thm:SingleStepMAP-PolyHyp,thm:FiniteMAP-poly} cannot be replaced, in general, by an arbitrary polyhedron. This is not even the case when the set $H$ in the theorem and corollary is affine only. Furthermore, this example satisfies local linear regularity, and thus this condition is not sufficient for finite convergence of MAP. 

\end{example}

The next example shows that BAP-EB is fulfilled for more problems  than  intrinsic transversality.

\begin{example}[BAP-EB is more general than intrinsic transversality]\label{ex:BAP-EB-Intrinsic}
  Let us look at the sets $X\coloneqq \{(x_{1},x_{2})\in\RR^{2} \mid x_{1}^2 - x_{2} + 1 \leq 0 \text{ and } -x_{1} - x_{2} + 1 \leq 0 \}$ and  $Y\coloneqq \{(y_{1},y_{2})\in\RR^{2} \mid y_{1}  + y_{2} \leq 0  \text{ and } y_{1}^2 + y_{2} \leq 0 \}$.  They have empty intersection, only a single best approximation pair is associated to them, namely $(x^{*},y^{*})$ with $x^{*}= (1,0)$ and $y^{*}= (0,0)$, and clearly BAP-EB from \cref{def:BAP-error-bound-II}  is satisfied. Thus, MAP applied to $X$ and $Y$ converges finitely.  However, intrinsic transversality fails bilaterally. In order to see that, take the sequences  determined by $x^{k}\coloneqq \left(\frac{1}{k},\frac{1}{k^{2}}+1 \right) \in X\backslash\{x^{*}\}$ and $y^{k}\coloneqq  \left(\frac{-1}{k},\frac{-1}{k^{2}  } \right) \in Y\backslash\{y^{*}\}$ and note that the unit vector $\left(\frac{-2}{\sqrt{k^{2} + 4}},\frac{k}{\sqrt{k^{2} + 4}}\right)$ lies in $ -\Ncone_{X}(x^k) \cap \Ncone_{Y}(y^k)$. Hence, intrinsic transversality does not hold, even if $X$ and $Y$ switch roles in \cref{def:intr-trans}. 

\end{example}

The aim of the following example is to show that BAP-EB in the original space is not necessarily transported to the product space. 
\begin{example}[BAP-EB in the original space and in the product space]\label{ex:Cimmino-MAP-ProductSpace}
  Consider the polyhedron $X\coloneqq \{(x_{1},x_{2})\in\RR^{2} \mid \abs{x_{1}}- x_{2} + 2 \leq 0  \}$, the polyhedron $Y\coloneqq \{(y_{1},y_{2})\in\RR^{2} \mid \abs{y_{1}}  + y_{2} \leq 0  \}$ and the hyperplane $H\coloneqq \{ (x_{1},x_{2}) \in\RR^{2} \mid  x_{2} = 0 \}$. Note that  $\bap(X,Y) = \bap(X,H) = \{(x^{*},y^{*})\}$, where  $x^{*} = (0,2)$ and $y^{*} = (0,0)$. MAP applied to either $H$ and $X$, or $Y$ and $X$ has finite convergence wherever one starts in $\RR^{2}$. Cimmino, however, only converges finitely to ${(x^{*}+y^{*})}/{2} = (0,1)$ when applied to the pair $Y$ and $X$. The relation between MAP and Cimmino discussed in \Cref{sec:finite-extensions} thus guarantees that MAP  converges finitely in Pierra’s product space regarding $Y$ and $X$.  Actually, the BAP-EB between $Y$ and $X$ is transported to  the product space. Nevertheless, BAP-EB holds for $H$ and $X$ but Cimmino does not converge finitely to $(0,1)$ for this pair of sets, which means, again due to the relation between MAP and Cimmino, that MAP does not converge in a finite number of steps in the product space, preventing BAP-EB to hold there.  

\end{example}


It is well-known that for the consistent case,  Lipschitzian/H\"older error bound gives us in general  linear/sublinear convergence of MAP \cite{Bauschke:1996,Drusvyatskiy:2016,Borwein:2014a,Gubin:1967,Borwein:2017}. With that said,  \Cref{thm:FiniteMAP-Convex} consists of a striking jump in the rate of convergence of MAP, namely from linear to finite. In this regard, the final example of this section incites the formulation of an intriguing question: Do we get linear convergence of MAP under infeasibility together with a H\"older error bound?


\begin{example}[A ball and a hyperplane in $\RR^{2}$]\label{ex:ballHyper}
  Consider the closed  ball $X$ in $\RR^2$ with radius $1$ centered in $(0,1)$, that is,   $X\coloneqq \{(x_{1},x_{2})\in\RR^{2}\mid x_1^2 + (x_2- 1)^2\leq 1\}$ and the family of horizontal hyperplanes of the form $Y_{\varepsilon} \coloneqq \{(t,-\varepsilon)\in \RR^{2}\mid t\in\RR \}$, where $\varepsilon$ is a nonnegative parameter. Note that $Y_0$ coincides with the optimal supporting hyperplane $\supH_{Y_{\varepsilon}}(X)$, for all positive $\varepsilon$. We have that $X\cap Y_0$ consists of the origin. If $\varepsilon>0$, $X\cap Y_\varepsilon=\emptyset$ and there exists a unique best approximation pair, namely $\left\{(0,0),(0,-\varepsilon)\right\}$.
  \begin{figure*}[!hbtp]\centering
  \centering
  \includegraphics[width=.58\textwidth]{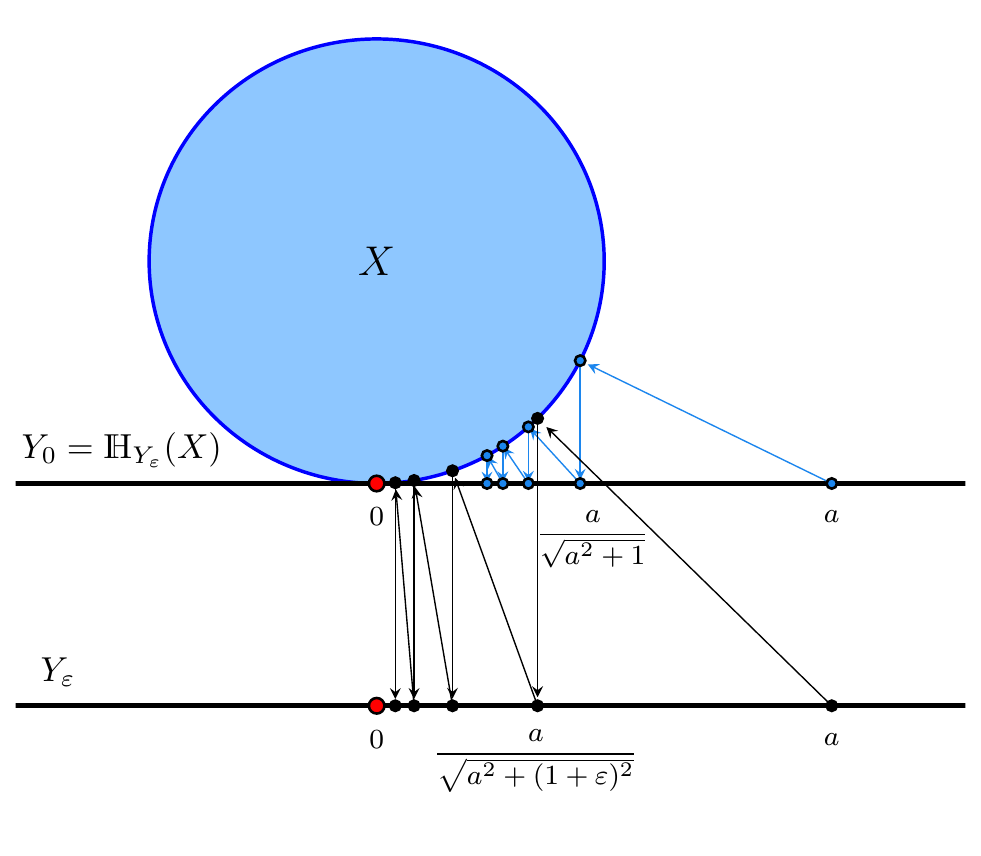}
\caption{Hölder error bound between  $X$ and $\supH_{Y_{\varepsilon}}(X)$.}
 \label{fig:MAPfigureExample}
\end{figure*}

We would like to draw our attention to a MAP iteration between $X$ and $Y_0$, and one concerning $X$ and $Y_\varepsilon$. These MAP iterations are given below and illustrated in \Cref{fig:MAPfigureExample}.
For a number  $a>0$, consider $(a,0)\in Y_0$ and  $(a,-\varepsilon)\in Y_\varepsilon$. Note that
\[
  P_{Y_0}P_X(a,0) = \left(\frac{a}{\sqrt{a^2+1}},0\right)\in Y_0\]
and  
  \[  P_{Y_\varepsilon}P_X(a,-\varepsilon) = \left(\frac{a}{\sqrt{a^2+(1+\varepsilon)^2}},-\varepsilon\right)\in Y_\varepsilon.
\]

Note also that 
\[\label{eq:limY0}
  \lim_{a\to0^+}  \frac{\norm{P_{Y_0}P_X(a,0) - (0,0)}}{\norm{(a,0) - (0,0)}} = \lim_{a\to0^+} \frac{1}{\sqrt{a^2+1}} = 1,
\]
and
\[\label{eq:limYeps}
  \lim_{a\to0^+}  \frac{\norm{P_{Y_\varepsilon}P_X(a,-\varepsilon) - (0,-\varepsilon)}}{\norm{(a,-\varepsilon) - (0,-\varepsilon)}} = \lim_{a\to0^+} \frac{1}{\sqrt{a^2+(1+\varepsilon)^2}} = \frac{1}{1+\varepsilon}.
\]
The limit being $1$ in \cref{eq:limY0} means that a MAP sequence between $X$ and $Y_0$, monitored on $Y_0$, converges to the origin sublinearly. On the other hand, in view of \cref{eq:limYeps}, the gap of size $\varepsilon$  leads to a linear convergence of a MAP sequence to the best point $(0,-\varepsilon)$  with asymptotic rate $\frac{1}{1+\varepsilon}$, where this MAP sequence regards $X$ and $Y_\varepsilon$ and belongs to $Y_\varepsilon$. For simplicity, we are looking at  MAP shadows on $Y_0$/$Y_\varepsilon$, but of course, due to the nonexpansiveness of  projections, the rates would be sublinear/linear for the correspondent shadows on $X$.

Apparently, this leap from sublinear to linear convergence occurs in view of the presence of a Hölder error bound condition between $X$ and $\supH_{Y_{\varepsilon}}(X)$.  In fact, we have the existence of   $\omega > 0$  such that  
\[\label{eq:error-bound-Holder}
\omega \dist(x,X \cap \supH_{Y_{\varepsilon}}(X))^{1+q} \leq \dist(x, \supH_{Y_{\varepsilon}}(X)),
\]
$x\in X$ in a neighborhood of $X\cap \supH_{Y_{\varepsilon}}(X) = \{(0,0)\}$, with $q=1$. 
We remark that the Hölder parameter $q=1$ can be effortlessly derived in this example by taking into account that   $X$ is entirely contained in the epigraph of the  quadratic function $g(t) = \frac{t^2}{2}$.

\end{example}

The previous example presented a particular instance where adding inconsistency led to a gain in the convergence rate of MAP from sublinear to linear. We understand this is not something one can expect without assuming some H\"older regularity. 
Probably, the improvement of MAP would be more limited if we considered MAP for a Cauchy bowl getting apart from a hyperplane. By Cauchy bowl \cite{Arefidamghani:2021} we mean something like the epigraph of the function $f:D\to \RR$ defined such that $f(0)=0$ and $f(x)=e^{{\norm{x}}^{-2}}$ elsewhere, {}where $D = \{x\in\RR^{n}\mid \norm{x} \leq 1/\sqrt{3}\}$. This function $f$ is known to be infinitely differentiable in the interior of its domain, but not analytic. Furthermore, its epigraph does not satisfy a H\"older error bound with respect to the hyperplane $H=\left\{(x,0)\in\RR^{n+1}\mid x\in\RR^n\right\}$.

We have just risen the conjecture of whether 
Hölder regularity, with $q=1$, provides linear convergence for MAP in the inconsistent case. This looks very plausible and if indeed correct, it can establish MAP as a practical and (perhaps) competitive  method for non-smooth optimization, at least for convex min-max problems featuring convex quadratic functions. Perhaps K{\L}-theory \cite{Attouch:2010,Li:2018a} may shed some light on the question. 

\begin{example}[MAP for convex min-max optimization]\label{ex:min-max}
Consider the following problem
\[\label{eq:minmax}
\min_{x\in \RR^n} g(x)\coloneqq \max \left\{ f_1(x),\ldots,f_m(x) \right\},
\]
where each $f_i:\RR^n\to \RR$ is a convex  function and assume that $g$ has a minimizer. Let $g^*$ be the optimal value of \cref{eq:minmax}, take $\beta < g^*$ and define the hyperplane $H_\beta\coloneqq\{(x,\beta)\in\RR^{n+1}\mid x\in\RR^n \}$. Thus, the epigraph of $g$, $\epi g\coloneqq\{(x,t)\in\RR^{n+1}\mid x\in\RR^n, g(x)\leq t  \}$, does not intersect $H_\beta$ and $\bap_{H_{\beta}}(\epi g)$ consists of the minimizers of $g$.  Moreover, due to \cref{conv-MAP}(i), for any starting point $(x^0,t_0)\in\RR^{n+1}$ the  MAP sequence defined by
\[\label{eq:seqMAPMinMax}
(x^{k+1},t_{k+1}) \coloneqq P_{\epi g}P_{H_\beta}((x^k,t_k))
\]
converges to a point $(x^*, g^*)$ such that $x^*$ is a solution of~\cref{eq:minmax}. 

Note that if all $f_i$ are affine, solving~\cref{eq:minmax} relates to linear programming, in the sense discussed in \cref{ex:MAP-LP}, and finite convergence of the MAP sequence  \cref{eq:seqMAPMinMax} is assured by~\cref{thm:FiniteMAP-poly}.

Now, \Cref{cor:FiniteMAP-ConvexHyperplane} opens the possibility for the MAP  sequence~\cref{eq:seqMAPMinMax} to achieve finite convergence in a more general setting. If, for instance,  the solution $x^*$ of \cref{eq:minmax} is unique and if the tangent cone of $\epi g$ at $(x^*,g^*)$ is pointed, the Lipschitzian BAP-EB \cref{eq:error-bound-condXH} holds and the MAP sequence \cref{eq:seqMAPMinMax} converges to $(x^*,g^*)$ after a finite number of steps.

 Of course that addressing a non-affine setting in practice relies on the calculation of projections onto epigraphs and a lower bound $\beta$ for $g$  in hand. Getting $\beta$ is potentially an easier task. For instance, if there is an index $\hat \imath$ for which $f_{\hat \imath}$ has a minimizer $\hat x$, we  have that any $\beta < f_{\hat \imath}(\hat x)$ provides a strict lower bound for $g$ as $f_{\hat \imath}(\hat x)\leq g^*$. 

 With that said, one can think of using MAP for quadratic min-max, that is, when all $f_i$'s are convex quadratic functions. In this case, finding $\beta$ is easy and projecting onto the epigraph of $g$ is somehow manageable \cite{Dai:2006}. As for the BAP-EB ~\cref{eq:error-bound-condXH}, it may or may not be satisfied. However, the Hölder condition \cref{eq:error-bound-Holder} holds with $q=1$, in view of the quadratic growth of $g$. Hence, if the conjecture raised in \cref{ex:ballHyper} is true for $q=1$, we would get either finite or linear convergence of the MAP sequence \cref{eq:seqMAPMinMax} for convex quadratic min-max problems. 

Furthermore, the conjecture being true would echo as well in min-max problems featuring strongly convex functions. The maximum of a finite number of strongly convex functions is strongly convex, has a unique minimizer and one can bound $g$ below by  a strictly convex quadratic function with minimum value $g^*$.  So,  H\"older regularity  holds with $q=1$ in the strongly convex setting of \cref{eq:minmax} and again, one would   
 either get finite or linear convergence of the MAP sequence \cref{eq:seqMAPMinMax}.

 The bonds between min-max problems and H\"older regularity together with the validity of the conjecture formulated in  \cref{ex:ballHyper} would be an asset in the field of non-smooth convex  optimization.

\end{example}

\section{Concluding remarks}\label{sec:concluding}

We have derived finite convergence of alternating projections for two non-intersecting closed convex sets satisfying a Lipschitzian error bound condition, which has been to proven be more general than the well known concept of intrinsic transversality. This result strengthens the theory on MAP, a widely acclaimed method in Mathematics. In addition to being interesting from a theoretical point of view, our main theorem may also have an impact on practical issues regarding projection-type algorithms in general, as inconsistency has been seen favorable for MAP. A question left open is to what extent MAP can improve when embedding inconsistency to a problem satisfying a Hölder error bound condition. As discussed within an example, a positive answer in this direction may turn MAP into an attractive method for minimizing the maximum of convex functions, a central problem in non-smooth optimization. Although our main results are for finite convergence of MAP, we have as well stated theorems on other projection/reflection-based methods for both two-set and multi-set inconsistent feasibility problems. In particular, the Cimmino method, the Douglas-Rachford method and Cyclic projections have been addressed. 


\section*{Acknowledgments} 
 The authors would like to thank the associate editor and the anonymous referees for their valuable suggestions, which significantly improved this manuscript. 

\bibliographystyle{siamplain}

\bibliography{refs}

\end{document}